\documentclass{amsart}

\usepackage{mathrsfs}
\usepackage{amsmath}
\usepackage{amssymb}
\usepackage{amsbsy}

\usepackage{amsfonts}

\newtheorem{theorem}{Theorem}[section]
\newtheorem{lemma}[theorem]{Lemma}
\theoremstyle{corollary}
\newtheorem{corollary}[theorem]{Corollary}
\newtheorem{example}[theorem]{Example}
\newtheorem{xca}[theorem]{Proposition}
\newtheorem{proposition}[theorem]{Proposition}

\theoremstyle{remark}

\numberwithin{equation}{section}

\def\a{\alpha} \def\b{\beta} \def\g{\gamma} \def\d{\delta} 
 \def\s{\sigma}   
 \def\ld{\lambda}

\def\G{\Gamma}

 \def\lg{\langle} \def\rg{\rangle}
\def\nd{\mathrel{\bigm|\kern-.7em/}}

\def\f{\noindent}

\def\PSL{\hbox{\rm PSL}}

\def\PGL{\hbox{\rm PGL}}

\def\Aut{\hbox{\rm Aut}}

\def\BiCay{\hbox{\rm BiCay}}
\def\Cay{\hbox{\rm Cay}}

\def\mz{{\mathbb Z}}

\def\K{{\bf K}}

\begin{document}

\title{Edge-transitive bi-Cayley graphs}

\author{Marston Conder}
\address{Department of Mathematics,
University of Auckland, Private Bag 92019, Auckland 1142, New Zealand}
\email{m.conder@auckland.ac.nz}

\author{Jin-Xin Zhou}
\address{Department of Mathematics,
Beijing Jiaotong University, Beijing 100044, P.R. China}
\email{jxzhou@bjtu.edu.cn}
\thanks{
The second author was partially supported by the National
Natural Science Foundation of China (11271012) and the Fundamental
Research Funds for the Central Universities (2015JBM110).}

\author{Yan-Quan Feng}
\address{Department of Mathematics,
Beijing Jiaotong University, Beijing 100044, P.R. China}
\email{yqfeng@bjtu.edu.cn}
\thanks{The third author was partially supported by the National Natural Science Foundation of China (11231008, 11571035).}

\author{Mi-Mi Zhang}
\address{Department of Mathematics,
Beijing Jiaotong University, Beijing 100044, P.R. China}
\email{14118412@bjtu.edu.cn}





\begin{abstract}
A graph $\G$ admitting a group $H$ of automorphisms acting semi-regularly on the vertices with exactly two orbits is called a {\em bi-Cayley graph\/} over $H$. Such a graph $\G$ is called {\em normal\/} if $H$ is normal in the full automorphism group of $\G$, and {\em normal edge-transitive\/} if the normaliser of $H$ in the full automorphism group of $\G$  is transitive on the edges of $\G$.
In this paper, we give a characterisation of normal edge-transitive bi-Cayley graphs, 
and in particular, we give a detailed description of $2$-arc-transitive normal bi-Cayley graphs. Using this, we investigate three classes of bi-Cayley graphs, namely those over abelian groups, dihedral groups and metacyclic $p$-groups.  We find that under certain conditions, `normal edge-transitive'  is the same as `normal' for graphs in these three classes.  As a by-product, we obtain a complete classification of all connected trivalent edge-transitive graphs of girth at most $6$, and answer some open questions from the literature about $2$-arc-transitive, half-arc-transitive and semisymmetric graphs.
\end{abstract}

\maketitle




\section{Introduction}
\label{sec:intro}

In this paper we describe an investigation of edge-transitive bi-Cayley graphs, in which we show that in the `normal' case, such graphs can be arc-transitive, half-arc-transitive, or semisymmetric, and we exhibit some specific families of examples of each kind.  As a by-product of this investigation, we answer two open questions about  edge-transitive graphs posed in 2001 by  Maru\v si\v c and Poto\v cnik  \cite{MP-gf}.
Before proceeding, we give some background to this topic, and set some notation.

\smallskip
First, if $G$ is a group acting on a set $\Omega$, then the stabiliser in $G$ of a point $\a \in \Omega$ is the subgroup $G_{\a} = \{ g \in G \mid \a^g = \a\}$ of $G$.  The group $G$ is said to be {\em semi-regular\/} on $\Omega$ if $G_\a=1$ for every $\a\in \Omega$, and {\em regular\/} on $\Omega$  if $G$ is transitive and semi-regular on $\Omega$.

A graph $\G$ is called a {\em Cayley graph\/} if it admits a group $G$ of automorphisms acting regularly on its vertex-set $V(\G)$.  In that case, $\G$ is isomorphic to the graph ${\rm Cay}(G,S)$ with vertex-set $G$ and edge-set $\{\{g,xg\} : g \in G, \, x \in S\}$, where $S$ is the subset of elements of $G$ taking the identity element to one of its neighbours (see \cite[Lemma~16.3]{B});  and then the automorphism group of $\G$ contains a subgroup $R(G) = \{R(g) : g\in G\},$ where $R(g)$ is right multiplication (the permutation of $G$ given by $R(g): x\mapsto xg$ for $x\in G$), for each $g \in G$.

If, instead, we require the graph $\G$ to admit a group $H$ of automorphisms acting semi-regularly on $V(\G)$ with two orbits, then we call $\G$ a {\em bi-Cayley graph\/} (for $H$).
In this case, $H$ acts regularly on each of its two orbits on $V(\G)$, and the two corresponding induced subgraphs are Cayley graphs for $H$.  In particular, we may label the vertices of these two subgraphs with elements of two copies $H_0$ and $H_1$ of $H$, and find that there are subsets $R$, $L$ and $S$ of $H$ (with $|R| = |L|$) such that the edges of those two induced subgraphs are of the form $\{h_0,(xh)_0\}$ with $h_0 \in H_0$ and $x \in R$, and $\{h_1,(yh)_1\}$ with $h_1 \in H_1$ and $y \in L$, while all remaining edges are of the form $\{h_0,(zh)_1\}$ with $z \in S$ and where $h_0$ and $h_1$ are the elements of $H_0$ and $H_1$ that represent a given $h \in H$.
This gives a concrete realisation of $\G$ in terms of $H$, $R$, $L$ and $S$.

Conversely, if $H$ is any group,  and $R$, $L$ and $S$ are subsets of $H$ with $|R| = |L|$ such that $1_H \not\in R = R^{-1}$ and $1_H \notin L = L^{-1}$, then the graph $\G$ with vertex set being the union $H_0 \cup H_1$ of two copies of $H$ and with edges of the form $\{h_0,(xh)_0\}$, $\{h_1,(yh)_1\}$ and $\{h_0,(zh)_1\}$ with $\,x \in R$, $\,y \in L$ and $z \in S$, and $h_0 \in H_0$ and  $\,h_1 \in H_1$ representing a given $h \in H$, is a bi-Cayley graph  for $H$.  Indeed $H$ acts as a semi-regular group of automorphisms by right multiplication, with  $H_0$ and $H_1$ as its orbits on vertices.
We denote this graph by $\BiCay(H,R,L,S)$, and denote the group of automorphisms induced by $H$ on the graph as $R(H)$.

\smallskip
Bi-Cayley graphs, which have sometimes been called {\em semi-Cayley\/} graphs, form a class of graphs that has been studied extensively --- as in \cite{Aral2,cao2009,dJ,Gao2011,Gao2010,HMP,KKMW,kovacs,LM,Liang2008,LWX,Luo2009,MMS,Mart2015,Mart,M,Pisanski}.
Note that some authors label the vertices of a bi-Cayley graph for a group $H$ with ordered pairs $(h,i)$ for $h \in H$ and $i \in \{0,1\},$ while we are using $h_i$ to denote $(h,i)$.

Various well known graphs can be constructed as bi-Cayley graphs.
For example, the famous Petersen graph is a bi-Cayley graph over a cyclic group of order $5$, and the Gray graph~\cite{Bouwer1968} (which is the smallest trivalent semisymmetric graph), is a bi-Cayley graph over a metacyclic group of order $27$.  Similarly, in Bouwer's answer \cite{Bouwer1970}  to Tutte's question~\cite{Tutte1966} about the existence of half-arc-transitive graphs with even valency, the smallest graph in his family is also a bi-Cayley graph over a  non-abelian metacyclic group of order $27$.
We note that all of these interesting small bi-Cayley graphs are  edge-transitive. This motivated us to investigate the class of all edge-transitive bi-Cayley graphs.

Although we know that a bi-Cayley graph $\G = \BiCay(H,R,L,S)$ has a `large' group of automorphisms, with just two vertex-orbits, it is difficult in general to decide whether a bi-Cayley graph is edge-transitive.  It helps to consider the subgroup $R(H)$ of the automorphism group $\Aut(\G)$ induced by $H$ on $\G$, and also its normaliser $N_{\Aut(\G)}(R(H))$.
The latter subgroup was characterised in \cite{ZF-auto} (see Proposition~\ref{normaliser} in the next section), making it  possible to determine whether or not $N_{\Aut(\G)}(R(H))$ is transitive on the edges of $\G$.
A bi-Cayley graph $\G=\BiCay(H, R, L, S)$ is called {\em normal\/} if $R(H)$ is normal in $\Aut(\G)$, and {\em normal edge-transitive\/} if $N_{\Aut(\G)}(R(H))$ is transitive on the edge-set of $\G$.

Similarly, we say that a bi-Cayley graph $\G=\BiCay(H, R, L, S)$ is {\em normal locally arc-transitive\/} if the stabilizer $(N_{\Aut(\G)}(R(H))_{1_0}$ of the `right' identity vertex $1_0$ in $N_{\Aut(\G)}(R(H))$ is transitive on the neighbourhood of $1_0$ in $\G$, and {\em normal half-arc-transitive} if $N_{\Aut(\G)}(R(H))$ is transitive on the vertex-set and edge-set of $\G$ but intransitive on the arc-set of $\G$.

These notions generalise the corresponding ones for Cayley graphs, which have also been studied extensively.
For example, a Cayley graph $\G=\Cay(G, S)$ is {\em normal edge-transitive\/} if $N_{\Aut(\G)}(R(G))$ is transitive on the edges of $\G$.  The study of such graphs was initiated in~\cite{Praeger-normal} by Praeger, and they play an important role in the study of edge-transitive graphs (as in \cite{Li-edge}, for example).

In this paper, we focus attention on normal edge-transitive bi-Cayley graphs.

Our motivation comes partly from the work of Li~\cite{CHL-pams} on bi-normal Cayley graphs. Let $\G=\Cay(G, S)$ be a Cayley graph on a group $G$, and let $R(G)\leq X\leq \Aut(\G)$. Then $\G$ is said to be {\em $X$-bi-normal} if the maximal normal subgroup $\bigcap_{x\in X} R(G)^x$ of $X$ contained in $R(G)$ has index $2$ in $R(G)$, and if $X=\Aut(\G)$, then $\G$ is called a {\em bi-normal Cayley graph}.  Note that by definition, a connected bi-normal Cayley graph $\G=\Cay(G, S)$ is a normal bi-Cayley graph over $H=\bigcap_{x\in X} R(G)^x$.

In~\cite[Question~1.2(a)]{CHL-pams}, Li asked whether there exist $3$-arc-transitive bi-normal Cayley graphs,
and in~\cite[Problem~1.3(b)]{CHL-pams}, he asked for a good description of 2-arc-transitive bi-normal Cayley graphs.
As we will see later, $R=L=\emptyset$ for every connected normal edge-transitive bi-Cayley graph $\BiCay(H,R,L,S)$,
and hence our first theorem below provides answers to these two questions.

\begin{theorem}\label{normal-2-arc}
Let $\,\G$  be a connected bi-Cayley graph $\BiCay(H,\emptyset,\emptyset,S)$.
Then $N_{\Aut(\G)}(R(H))$ acts transitively on the $2$-arcs of $\,\G$ if and only if the following three conditions hold$\,:$
\begin{enumerate}
\item[{\rm (a)}]\ there exists an automorphism $\a$ of $H$ such that $S^\a=S^{-1}$,
\\[-10pt]
\item[{\rm (b)}]\ the setwise stabiliser of $S\setminus\{1\})$ in $\Aut(H)$ is transitive on $S\setminus\{1\}$, and
\\[-10pt]
\item[{\rm (c)}]\ there exists $s\in S\setminus\{1\}$ and an automorphism $\b$ of $H$ such that $S^\b=s^{-1}S$.
\end{enumerate}
Furthermore, $N_{\Aut(\G)}(R(H))$ is not transitive on the $3$-arcs of $\,\G$.
\end{theorem}

Next, it is natural to consider the classification of edge-transitive graphs into three distinct types: arc-transitive graphs, half-arc-transitive graphs (which are vertex- and edge- but not arc-transitive), and semisymmetric graphs (which are edge- but not vertex-transitive). We show that there are normal edge-transitive bi-Cayley graphs of each type:

\begin{theorem}\label{normal-edge}
If $\,\G$ is a normal edge-transitive bi-Cayley graph, then $\,\G$ can be either arc-transitive, half-arc-transitive or semisymmetric.  Furthermore, infinitely many examples exist in each case.
\end{theorem}

Note that this contrasts with the situation for Cayley graphs, which are vertex-transitive and therefore cannot be semisymmetric.
To prove Theorem~\ref{normal-edge}, we need only prove the assertion about the existence of the graphs in each case, and to do that, we consider bi-Cayley graphs over abelian groups, dihedral groups and metacyclic $p$-groups.
Any such graph may be called {\em bi-abelian graph},  or a {\em bi-dihedrant}, or  a {\em bi-$p$-metacirculant}, respectively.
In the bi-abelian case, we have the following:

\begin{xca}\label{bi-abelian}
Every connected bi-Cayley graph $\,\G = \BiCay(H, R, L, S)$ over an abelian group is vertex-transitive.
Moreover, if $\,\G$ is half-arc-transitive, then $R\cup L$ is non-empty and does not contain any involution,
$|R|=|L|$ is even, $|S|>2$, and the valency of $\,\G$ is at least $6$.
\end{xca}

We then use this to study trivalent edge-transitive graphs with small girth, motivated by the work of Conder and Nedela \cite{CD} and Kutnar and Maru\v si\v c ~\cite{Kutnar-JCTB} on classification of connected trivalent arc-transitive graphs of small girth.
We obtain the following generalisation, to all connected trivalent edge-transitive graphs of girth $6$ (noting that a trivalent edge-transitive graph is either arc-transitive or semisymmetric).

\begin{theorem}\label{th-girth}
All connected trivalent edge-transitive graphs of girth at most $6$ are known, and in particular, every connected trivalent semisymmetric graph has girth at least $8$.  
\end{theorem}

Trivalent semisymmetric graphs of girth $8$ are known to exist, and include the Gray graph; see \cite{Semi-symm-768}.

Our semisymmetric normal edge-transitive bi-Cayley graphs are constructed from bi-dihedrants. The motivation for us to consider semisymmetric bi-dihedrants is the work of Maru\v si\v c and Poto\v cnik~\cite{MP-gf} on worthy semisymmetric tetracirculants. A graph is called {\em tetracirculant\/} if its automorphism group contains a cyclic semi-regular subgroup with four orbits, and a graph is said to be {\em worthy\/} if no two of its vertices have exactly the same set of neighbours.
Maru\v si\v c and Poto\v cnik proposed two questions (Problems~4.3 and 4.9 in \cite{MP-gf}) regarding the existence of worthy semisymmetric tetracirculants. Every bi-dihedrant is a tetracirculant, and an unworthy graph cannot be edge-regular, so our next theorem provides a positive answer to each of those.

\begin{theorem}\label{th-bi-dihedrant}
Every connected semisymmetric bi-dihedrant has valency at least $6$, and examples of semisymmetric bi-dihedrants of valency $2k$ exist for each odd integer $k \ge 3$. In particular, there exists a family of edge-regular semisymmetric bi-dihedrants of valency $6$.
\end{theorem}

Next, a {\em bi-$p$-metacirculant\/} is a bi-Cayley graph over a metacyclic $p$-group.
(A group $G$ is metacyclic if it has a normal subgroup $N$ such that both $N$ and $G/N$ are cyclic.)
For example, the smallest graph in a family of edge-transitive graphs constructed by Bouwer \cite{Bouwer1970}
is a tetravalent half-arc-transitive bi-Cayley graph over a metacylic group of order $27$.
This motivated us to consider tetravalent half-arc-transitive bi-$p$-metacirculants in general,
leading to our final theorem below.

\begin{theorem}\label{th-normal-bi-meta}
If $\,\G$ is a tetravalent vertex- and edge-transitive bi-Cayley graph over a non-abelian metacyclic $p$-group $H$, for some odd prime $p$, and $R(H)$ is a Sylow $p$-subgroup of $\Aut(\G)$, then $R(H)$ is normal in $\Aut(\G)$, and so $\G$ is a normal bi-Cayley graph. Indeed there exist  such tetravalent half-arc-transitive bi-$p$-metacirculants, for every odd prime~$p$.
\end{theorem}

Using this theorem, a complete classification of tetravalent half-arc-transitive bi-$p$-metacirculants will be given
in \cite{Zhang-Zhou}.  Also when proving Theorem~\ref{th-normal-bi-meta},
we were led to study a general question posed in 2008 by Maru\v si\v c and \v Sparl~\cite[p.368]{MS}
about metacirculants, and we give a positive answer to their question
(and point out an error in a paper by Li, Song and Wang \cite{LWS} that claimed to do the same thing).

\section{Preliminaries}
\label{sec:prelims}

Throughout this paper, groups are assumed to be finite, and graphs are assumed to be finite, connected, simple and undirected.
For the group-theoretic and graph-theoretic terminology not defined here, we refer the
reader to \cite{BMBook,WI}.  We proceed by introducing some basic concepts and terminology.

\subsection{Definitions and notation}
For a finite, simple and undirected graph $\G$, we use $V(\G)$, $E(\G)$, $A(\G)$ and $\Aut(\G)$ to denote the vertex-set, edge-set, arc-set and full automorphism group of $\G$, respectively. 
We let $d(u,v)$ be the distance between vertices $u$ and $v$ in $\G$,
and let $D$ be the {\em diameter\/} of $\G$, which is the largest distance between two vertices in $\G$.
For any vertex $v$ of $\G$, we let $\G(v)$ be the neighbourhood of $v$, and
more generally, we define $\G_i(v)=\{u \mid d(u,v)=i\}$ (the set of vertices at distance $i$ from $v$), for $1\le i\le D$.
For any subset $B$ of $V(\G)$, the subgraph of $\G$ induced by $B$ is denoted by $\G[B]$,
and the neighbourhood of $B$ in $\G$ is defined as $N_\G(B)=\bigcup_{v\in B}(\G(v)) \setminus B$.

For each non-negative integer $s$, an {\em $s$-arc\/} in $\G$ is an ordered $(s+1)$-tuple $(v_0,v_1, \cdots ,v_{s-1},v_s)$ of vertices of $\G$ such that $v_{i-1}$ is adjacent to $v_i$ for $1\leq i\leq s$, and $v_{i-1} \neq v_{i+1}$ for $1 \leq i \leq s-1$, or in other words, such that any two consecutive $v_j$ are adjacent and any three consecutive $v_j$ are distinct.
In particular, a $0$-arc is a vertex, and a $1$-arc is usually called an {\em arc}.
The graph $\G$ is said to be {\em $s$-arc-transitive\/} if $\Aut(\G)$ is
transitive on the set of all $s$-arcs in $\G$. In particular, $0$-arc-transitive means {\em vertex-transitive}, and $1$-arc-transitive means {\em arc-transitive}, or {\em symmetric}.
Similarly, $\G$ is {\em $s$-arc-regular\/} if $\Aut(\G)$ acts regularly (sharply-transitively) on the set of all $s$-arcs in $\G$. In particular, $\G$ is $1$-arc-regular, or simply {\em arc-regular}, if given any two arcs of $\G$, there exists a unique automorphism of $\G$ taking one arc to the other.

A graph $\G$ is {\em edge-transitive\/} or {\em edge-regular\/} if $\Aut(\G)$ acts transitively or regularly on $E(\G)$, respectively,  and $\G$ is {\em semisymmetric\/} if $\G$ has constant valency and is edge- but not vertex-transitive, while $\G$ is {\em half-arc-transitive\/} if $\G$ is vertex-transitive and edge-transitive but not arc-transitive.

We use $C_n$ to denote the multiplicative cyclic group of order $n$, and $\mz_{n}$ for the ring 
of integers mod $n$,
and $\mz_{n}^{\,*}$ for the multiplicative group of units mod $n$ (the elements coprime to $n$ in $\mz_n$).
Also we use $D_n, A_n$ and $S_n$ respectively for the dihedral, alternating and symmetric groups of degree $n$.
For two groups $M$ and $N$, we use $N\rtimes M$ to denote a semi-direct product of $N$ by $M$ (with kernel $N$ and complement $M$).
Finally, for a subgroup $H$ of a group $G$, we denote by $C_{G}(H)$ and $N_{G}(H)$ respectively the centraliser and normaliser of $H$ in $G$, and for a permutation group $G$ on a set $\Omega$, we let $G_{(\Delta)}$ and $G_{\Delta}$
be respectively the pointwise and setwise stabiliser of a subset $\Delta$ of $\Omega$.

\subsection{Basic properties of bi-Cayley graphs}
In this subsection, we let $\G$ be a connected bi-Cayley graph $\BiCay(H,R,L,S)$
over a group $H$. It is easy to prove some basic properties of such a $\G$, as in \cite{ZF-auto}:

\begin{xca}\label{n}
 The following hold for any connected bi-Cayley graph $\BiCay(H,R,L,S)\!:$
\begin{enumerate}
\item[{\rm (a)}] $H$ is generated by $R\cup L\cup S;$
\item[{\rm (b)}] $S$ can be chosen to contain the identity of $H$ $($up to graph isomorphism$);$
\item[{\rm (c)}] $\BiCay(H,R,L,S)\cong \BiCay(H,R^{\alpha},L^{\alpha},S^{\alpha})$ for every automorphism $\alpha$ of $H;$ and
\item[{\rm (d)}] $\BiCay(H,R,L,S) \cong \BiCay(H,L,R,S^{-1})$.
\end{enumerate}
\end{xca}

We will say the triple $(R, L, S)$ of subsets of $H$ is a {\em bi-Cayley triple\/} if $R=R^{-1},  L=L^{-1}$ and $1\in S,$
and we will say that the two bi-Cayley triples $(R, L, S)$ and $(R', L', S')$ for the same group $H$ 
are {\em equivalent}, and write $(R, L, S)\equiv(R', L', S')$, if either $(R', L', S')=(L, R, S^{-1}),$ or $(R', L', S')=(R, L, S)^\a$
for some automorphism $\a$ of $H$. Note that by parts (c) and (d) of Proposition~\ref{n}, the bi-Cayley graphs for any two equivalent bi-Cayley triples are isomorphic.

\vskip2mm
Now we consider the automorphisms of the bi-Cayley graph $\G= \BiCay(H, R, L, S)$.
Recall that $H$ acts as a semi-regular group of automorphisms of $\G$ by right multiplication,
with  $H_0$ and $H_1$ as its orbits on vertices.
Indeed each $g\in H$ induces an automorphism $R(g)$ of $\G$ given by
$\,h_{i}^{R(g)}=(hg)_{i}\,$ for $ i \in \{0,1\}$ and $h \in H$,
and then $R(H)=\{R(g)\ |\ g\in H\}\leq\Aut(\G)$.
Next, for any automorphism $\a$ of $H$ and any elements $x,y,g\in H$,
we may define two permutations $\d_{\a,x,y}$ and $\s_{\a,g}$ on $V(\G)=H_0\cup H_1$ as follows:
\begin{equation}\label{2}
\begin{array}{ll}
\d_{\a,x,y}:& h_0\mapsto (xh^{\a})_1 \ \hbox{ and } \  h_1\mapsto (yh^{\a})_0, \ \hbox{ for each } h\in H,\\[+2pt]
{}\hskip 6pt \s_{\a,g}:& h_0\mapsto (h^\a)_0 \hskip 6pt   \ \hbox{ and } \ h_1\mapsto (gh^{\a})_1, \ \hbox{ for each } h\in H,\\
\end{array}
\end{equation}
and then define \begin{equation}\label{3}
\begin{array}{ll}
{\rm I} \, = \, \{\d_{\a,x,y}\ |\ R^\a=x^{-1}Lx, \ L^\a=y^{-1}Ry \, \hbox{ and } \, S^\a=y^{-1}S^{-1}x\},\\[+4pt]
\hbox{and } \ {\rm F} \, = \, \{ \s_{\a,g}\ |\ R^\a=R, \ L^\a=g^{-1}Lg \, \hbox{ and } \, S^\a=g^{-1}S\}.
\end{array}
\end{equation}

With the above notation, the proposition below is easy to prove.

\begin{xca}{\rm\cite[Theorem~1.1]{ZF-auto}}\label{normaliser} \
Let $\,\G=\BiCay(H,R,L,S)$ be a connected bi-Cayley graph over the group $H$.
Then $N_{\Aut(\G)}(R(H))=R(H)\rtimes {\rm F}$ if $\,{\rm I}=\emptyset,$ and
$N_{\Aut(\G)}(R(H))=R(H)\langle {\rm F},\delta_{\a,x,y}\rangle$ if
$\,{\rm I}\neq\emptyset$ and $\delta_{\a,x,y}\in {\rm I}$.
Moreover, for every $\delta_{\a,x,y}\in {\rm I}$ the following hold$\,:$
\begin{enumerate}
\item[{\rm (a)}] $\langle R(H),\delta_{\a,x,y}\rangle$ acts transitively on $V(\G),$ and \\[-10pt]
\item[{\rm (b)}] if $\a$ has order $2$ and $x=y=1$, then $\,\G$ is isomorphic to the Cayley graph $\Cay(H\rtimes \langle\a\rangle, R\cup \alpha S)$.
\end{enumerate}
\end{xca}

\section{General theory on normal edge-transitive Cayley graphs}
\label{sec:general}

\subsection{General properties}
We begin this section with the following lemma, which shows that every normal edge-transitive bi-Cayley graph is bipartite.

\begin{lemma}\label{is-bi-part}
Let $\,\G=\BiCay(H,R,L,S)$ be a connected normal edge-transitive bi-Cayley graph over the group $H$. Then $R=L=\emptyset$, and hence $\,\G$ is bipartite, with the two orbits of $R(H)$ on $V(\G)$ as its parts.
\end{lemma}

\begin{proof}
Let $X=N_{\Aut(\G)}(R(H))$, which is edge-transitive on $\G$. Now suppose $R\neq \emptyset$.  Then $\{1_0, r_0\}$ is an edge of $\G$, for some $r\in R$, and by edge-transitivity of $X$ on $\G$, we know that some element $g\in X$ takes $\{1_0, 1_1\}$ to $\{1_0, r_0\}$.  Also $H_0$ is an orbit of $R(H)$, and $R(H)$ is normal in $X$, and therefore $H_0$ is a block of imprimitivity for $X$ on $V(\G)$.  But then since $1_0^{\, g}\in\{1_0, r_0\}$ we have $H_0^{\, g}=H_0$, while on the other hand,
since $1_1^{\, g}\in\{1_0, r_0\}$ we have $H_1^{\, g}=H_0$, and therefore $H_0 = H_1$, contradiction. Thus $R=\emptyset$.
Similarly, $L=\emptyset$, and the rest follows easily.
\end{proof}

The next lemma shows that a normal bi-Cayley graph cannot be $3$-arc-transitive.

\begin{lemma}\label{no-3-arc-tran}
Let $\,\G=\BiCay(H,\emptyset,\emptyset,S)$ be a connected bi-Cayley graph, 
and let $X=N_{\Aut(\G)}(R(H))$. Then $X_{1_01_1}=\lg \, \s_{\a, 1}\ |\ \a\in\Aut(H, S\setminus\{1\})\,\rg$,
and hence $X$ does not act transitively on the $3$-arcs of $\,\G$.
\end{lemma}

\begin{proof}
By Proposition~\ref{normaliser}, we know that $X_{1_0}=\lg \, \s_{\a, s}\ |\ \a\in\Aut(H), \, s\in S, \, S^\a=s^{-1}S\rg$,
and hence that $X_{1_01_1}=\lg\,\s_{\a, 1}\ |\ \a\in\Aut(H), \, S^\a=S \, \rg=\lg\, \s_{\a, 1}\ |\ \a\in\Aut(H, S\setminus\{1\})\,\rg$.
Next, if $X$ acts transitively on the $3$-arcs of $\G$, then $X_{1_11_0s_1}$ acts transitively on $\G(1_1)\setminus\{1_0\}$, for some $s\in S\setminus\{1\}$.
But for any $\s_{\a, 1}\in X_{1_11_0s_1}$, we have $s_1=s_1^{\s_{\a,1}}=1\cdot (s^\a)_1=(s^\a)_1$, and so $s^\a=s$,
and then $(s^{-1})_0^{\s_{\a, 1}}=((s^{-1})^\a)_0=(s^{-1})_0$, and this contradicts the transitivity of $X_{1_11_0s_1}$ on $\G(1_1)\setminus\{1_0\}$.
\end{proof}

\subsection{Normal locally arc-transitive bi-Cayley graphs}
The following proposition gives a characterisation of normal locally arc-transitive bi-Cayley graphs.

\begin{xca}\label{normal-arc-tran}
Let $\,\G=\BiCay(H,\emptyset,\emptyset,S)$ be a connected bi-Cayley graph, 
and let $X=N_{\Aut(\G)}(R(H))$.
Then $\,\G$ is normal locally arc-transitive if and only if $\,\G(1_0)=\{s_1 : s \in S\}$ is an orbit
of the subgroup ${\rm F} = \{ \s_{\a,g}\ |\ S^\a=g^{-1}S\}$.
Moreover, if $\,\G$ is normal locally arc-transitive, then
\begin{enumerate}
\item[{\rm (a)}]\ $X$ acts transitively on the arcs of $\,\G$ if and only if there exists $\a\in\Aut(H)$ such that $S^\a=S^{-1}$;

\item[{\rm (b)}]\ $X$ acts semisymmetrically on $\,\G$ if and only if there exists no $\a\in\Aut(H)$ such that $S^\a=S^{-1}$.
\end{enumerate}
\end{xca}

\begin{proof}
The first assertion is obvious,  so now suppose that $\G$ is normal locally arc-transitive.
Then all we have to do is prove that $X$ acts transitively on the arcs of $\G$ if and only if $S^\a=S^{-1}$ for some $\a\in\Aut(H)$.
If $X$ acts transitively on the arcs of $\G$, then there exists $\d_{\a, x, y}\in{\rm I}$
such that $(1_0, 1_1)^{\d_{\a,x,y}}=(1_1,1_0)$, and then $1_1=1_0^{\d_{\a,x,y}}=(x\cdot 1^\a)_1=x_1$ and $1_0=1_1^{\d_{\a,x,y}}=(y\cdot 1^\a)_0=y_0$, and it follows that $x=y=1$, so $\d_{\a,1,1}\in {\rm I}$,
and therefore $S^\a=S^{-1}$ (by (\ref{3})).
Conversely, if $S^\a=S^{-1}$ for some $\a\in\Aut(H)$, then $\d_{\a,1,1}$ takes $s_0$ to $(s^{-1})_1$,
and it follows that $X$ is vertex-transitive on $\G$, and therefore arc-transitive on $\G$.
\end{proof}

Recall that if a connected Cayley graph $\G=\Cay(G, S)$ on a group $G$ is bi-normal and arc-transitive, then $\G$ is a normal arc-transitive bi-Cayley graph over $H=\bigcap_{x\in X} R(G)^x$. Conversely, by Propositions~\ref{normal-arc-tran} and \ref{normaliser}, every arc-transitive normal bi-Cayley graph must be a Cayley graph.
On the other hand, an arc-transitive normal bi-Cayley graph is not necessarily a bi-normal Cayley graph.
\medskip\smallskip

We can now prove our first main theorem. \medskip

\f{\bf Proof of Theorem~\ref{normal-2-arc}}.
For necessity, suppose that $X =N_{\Aut(\G)}(R(H))$ acts transitively on the $2$-arcs of $\G$.
Then $X$ acts acts transitively on the arcs of $\G$, and so (a) holds, by Proposition \ref{normal-arc-tran}.
Also $X_{1_01_1}$ acts transitively on $\G(1_0)\setminus\{1_1\}$, and since  $X_{1_01_1}=\lg\, \s_{a,1}\ |\ \a\in\Aut(H, S)\rg$
by Lemma~\ref{no-3-arc-tran}, it follows that for all $s, t\in S\setminus\{1\}$, there exists $\s_{\a,1}\in X_{1_01_1}$ such that $t_1=s_1^{\s_{\a,1}}=(1\cdot s^\a)_1=(s^\a)_1$, and so $t=s^\a$. Thus $\Aut(H, S\setminus\{1\})$ is transitive on $S\setminus\{1\}$, proving (b).  Furthermore, for any $s\in  S\setminus\{1\}$ there exists $\s_{\b, s}\in X_{1_0}$ such that $1_1^{\s_{\b, s}}=s_1$, and then since $\s_{\b, s}\in X_{1_0}={\rm F}$, we must have $S^\b=s^{-1}S$, so (c) holds.

For sufficiency, suppose the conditions (a) to (c) hold.
Then by (a) and Proposition~\ref{normaliser}, we find that $X$ is vertex-transitive on $\G$,
and by (b) and (c), we find that $X_{1_0}$ is $2$-transitive on $\G(1_0)$.
Thus $X$ acts transitively on the $2$-arcs of $\G$.
\hfill\qed

\subsection{Normal half-arc-transitive bi-Cayley graphs}
\begin{xca}\label{normal-half-arc}
Let $\,\G=\BiCay(H,\emptyset,\emptyset,S)$ be a connected bi-Cayley graph, 
and let $X=N_{\Aut(\G)}(R(H))$. Then $X$ acts transitively on the vertices and edges but not on the arcs of $\,\G$ if and only if
\begin{enumerate}
\item[{\rm (a)}]\ $X_{1_0}$ has exactly two orbits on $\,\G(1_0)$ of equal size, say $O_1=1_1^{X_{1_0}}, O_2=x_1^{X_{1_0}}$,
and
\item[{\rm (b)}]\ there exists $\a\in\Aut(H)$ and $S^\a=S^{-1}x$.
\end{enumerate}
\end{xca}

\begin{proof}
Suppose $X$ acts transitively on the vertices and edges but not on the arcs of $\G$.
Then for any $s\in S$, we have  $X_{1_0s_1}=X_{\{1_0,s_1\}}$, so $|X: X_{1_0s_1}|=|X: X_{\{1_0, s_1\}}|=|E(\G)|=|V(G)||\G(1_0)|/2$, and therefore $|X_{1_0}: X_{1_0s_1}|=|\G(1_0)|/2$.  Thus $X_{1_0}$ has exactly two orbits on $\G(1_0)$ of equal size, and (a) holds. Now choose $x$ so that $1_1$ and $x_1$ lie in different orbits of $X_{1_0}$ on $\G(1_0)$. Then there exists $\d_{\a,a,b}\in X$ such that $\{1_0, 1_1\}^{\d_{\a, a, b}}=\{1_0, x_1\}$, and then $1_0^{\d_{\a,a,b}}=x_1$ and $1_1^{\d_{\a, a, b}}=1_0$, so that $a=x$ and $b=1$. In this case $\a\in\Aut(H)$, and $S^\a=S^{-1}x$, so (b) holds.

Conversely, suppose (a) and (b) hold. Then $X$ is vertex-transitive on $\G$, by (b), but not arc-transitive on $\G$, by (a).
Next, for any edge $\{h_0, g_1\}$, we have $\{h_0, g_1\}^{R(h^{-1})}=\{1_0, (gh)_1\}$,
and by (a) it follows that $(gh)_1\in O_1$ or $O_2$. If $(gh)_1\in O_1$, then clearly $\{h_0, g_1\}$ lies in the same orbit of $X$
as $\{1_0, 1_1\}$.  On the other hand, if $(gh)_1\in O_2$, then there exists $\s\in X_{1_0}$ such that $(gh)_1^\s=x_1$,
and then by (b) it follows that $\d_{\a, x, 1}$ is an automorphism of $\G$ with $(1_0, 1_1)^{\d_{\a, x, 1}}=(x_0, 1_1)$,
and then $\{1_0, (gh)_1\}^{\s\d_{\a,x,1}}=\{1_0, 1_1\}$, so again, $\{h_0, g_1\}$ lies in the same orbit of $X$ as $\{1_0, 1_1\}$.
Thus $X$ is edge-transitive on $\G$.
\end{proof}

Here we remark that in contrast to normal arc-transitive bi-Cayley graphs, normal half-arc-transitive bi-Cayley graphs may be non-Cayley; see Section~\ref{sec:tetravalent}.

\section{Edge-transitive bi-abelian graphs}
\label{sec:ET}

In this section, we prove Proposition~\ref{bi-abelian}. We do this in two steps.

\begin{xca}\label{abelian}
Let $\,\G=\BiCay(H, R, L, S)$ be a connected bi-Cayley graph over an abelian group $H$.
Then the following hold$\,:$
\begin{enumerate}
  \item [{\rm (a)}]\ If $R=L=\emptyset$, then $R(H)\rtimes\lg\d_{\a, 1, 1}\rg$ is regular on $V(\G)$, \\
  where $\a$ is the automorphism of $H$ that maps every element of $H$ to its inverse.
  \item [{\rm (b)}]\ If $\,\G$ is edge-transitive, then $\,\G$ is vertex-transitive.
\end{enumerate}
\end{xca}

\begin{proof}
Suppose $R=L=\emptyset$. Since $H$ is abelian, there exists an automorphism $\a$ of $H$ such that
$\a$ maps every element of $H$ to its inverse, and in particular, $S^\a=S^{-1}$.
It then follows from Proposition~\ref{normaliser} that $\d_{\a, 1, 1}$ is an automorphism of $\G$ of order $2$ interchanging
$H_0$ and $H_1$, and $R(H)\rtimes\lg\d_{\a, 1, 1}\rg$ is regular on $V(\G)$, which proves (a).
Next, for (b), suppose $\G$ is edge- but not vertex-transitive. Then $\G$ is semisymmetric, and
hence bipartite, with its parts being the two orbits of $\Aut(\G)$ on $V(\G)$.
It follows that $H_0$ and $H_1$ are two partition sets of $\G$, and so $R=L=\emptyset$,
but then by (a), $\G$ is vertex-transitive after all, contradiction. Thus $\G$ is vertex-transitive.
\end{proof}


\begin{xca}\label{4-valent abelian}
Let $\,\G=\BiCay(H, R, L, S)$ be a connected half-arc-transitive bi-Cayley graph over an abelian group $H$.
Then the following hold$\,:$
\begin{enumerate}
  \item [{\rm (a)}]\ $R\cup L$ is non-empty and contains no involution.
  \item [{\rm (a)}]\ $|R|=|L|$ is even, and $|S|>2$.
  \item [{\rm (c)}]\ $\G$ has valency $6$ or more.
\end{enumerate}
\end{xca}

\begin{proof}
Again let $\a$ be the automorphism of $\G$ that takes every element of $H$ to its inverse.
If $R=L=\emptyset$, then by Proposition~\ref{bi-abelian}, $\d_{\a,1,1}\in\Aut(\G)$ and $(1_0, 1_1)^{\d_{\a,1,1}}=(1_1,1_0)$, which implies that $\G$ is arc-transitive, contradiction. Hence $R\cup L$ is non-empty.  Also if $R$ contains an involution $h$, then $\{1_0, h_0\}\in E(\G)$ and $(1_0, h_0)^{R(h)}=(h_0, 1_0)$, which again implies that $\G$ is arc-transitive, contradiction. Similarly, $L$ does not contain an involution. This proves (a). Moreover, because $R=R^{-1}$ and $L=L^{-1}$, it also implies that $|R|=|L|$ is even. Next, since $\G$ is half-arc-transitive, the valency $|\G(1_0)|=|R|+|S|$ is even, and so $|S|$ is even, and therefore $|S|\geq 2$. Also by half-arc-transitivity, $\Aut(\G)_{1_0}$ has two orbits on $\G(1_0)$ of equal size, say $B_1$ and $B_2$, with $r_0$ and $(r^{-1})_0$ being in different orbits, for any $r\in R$. It follows that half of the elements of $R$ are contained $B_1$, and half are in $B_2$. But now suppose $|S|=2$, say $S=\{1, s\}$.  Then since $|B_1| = |B_2| = (|R|+|S|)/2 = |R|/2+1$,  the other two neighbours $1_1$ and $s_1$ of $1_0$ must lie in different orbits of $\Aut(\G)_{1_0}$. On the other hand, it is easy to check that $\s_{\a, s}\in\Aut(\G)_{1_0}$ takes $1_1$ to $s_1$, contradiction. Hence $|S|>2$, which proves (b).
Finally, this implies that the valency $|R|+|S|$ of $\G$ is at least $2+4 = 6$, proving (c).
\end{proof}

In fact, $6$ is the minimum valency of all connected half-arc-transitive bi-Cayley graphs over an abelian group,
and is achieved by the following example:

\begin{example}\label{exam-half-abelian}
{\rm Let $\G=\BiCay(H, R, L, S)$, where $H = \lg a\rg = C_{28}$,
and $R=\{a,a^{-1}\}$, $L=\{a^{13}, a^{-13}\}$ and $S=\{1,a,a^6,a^{19}\}$. Then $\G$ has valency $6$, and an easy computation using Magma~\cite{BCP} shows that $\G$ is half-arc-transitive, with $\Aut(\G)\cong (C_7\times Q_8)\rtimes C_3$, where $Q_8$ is the quaternion group.}
\end{example}

To complete this section, we give two easy corollaries of the above two propositions.

\begin{corollary}\label{abelian is not semisy}
No connected bi-Cayley graph over an abelian group is semisymmetric.
\end{corollary}

\begin{corollary}\label{cyclic}
Let $\,\G=\BiCay(H, \emptyset, \emptyset, S)$ be a connected trivalent edge-transitive bi-Cayley graph over a cyclic group $H\cong C_n$.
Then $n=2$ or $4$, or $n$ is a divisor of $k^2+k+1$ for some $k\in\mz_n^{\,*}$.
Furthermore, if $n\geq 13$ then $\,\G$ is $1$-arc-regular.
\end{corollary}

\begin{proof} By Proposition~\ref{abelian}, we know that $\G$ is an arc-transitive Cayley graph over
the group $R(H)\rtimes\lg\d_{\a, 1, 1}\rg$, which is dihedral of order $2n$.
It then follows from a theorem in \cite{MP1} on Cayley graphs over dihedral groups that $n=2$ or $4$,
or $n\ |\ k^2+k+1$ for some $k\in\mz_n^{\,*}$.
If $n=2$ or $4$, then $\G$ is isomorphic to the complete graph $K_4$ or the cube graph $Q_3$,
while if $n=3$ or $7$ then $\G$ is isomorphic to the complete bipartite graph $K_{3, 3}$
or the Heawood graph, and in all other cases (with $k \ge 3$ and $n\geq 13$), $\G$ is $1$-arc-regular.
\end{proof}

\section{Trivalent edge-transitive graphs with girth at most $6$}
\label{sec:ET3}

The aim of this section is to give a classification of all connected trivalent edge-transitive graphs with girth $6$,
and to prove Theorem~\ref{th-girth}.
We achieve this in two stages, the first being a special case, and the second the general case.

\subsection{Trivalent normal edge-transitive bi-abelian graphs}
We begin by defining a family of connected trivalent edge-transitive bi-abelian graphs.

Let $n$ and $m$ be any two positive integers with $nm^2\geq 3$.
If $n = 1$ take $\ld=0$, while if $n > 1$ take $\ld\in\mz_n^{\,*}$ such that $\ld^2-\ld+1\equiv 0$ mod $n$.
Now define
\begin{equation}\label{abelian groups}
\begin{array}{l}
\G_{m,n,\ld}=\BiCay(H,\emptyset, \emptyset,\{1, x, x^\ld y\}),
\ \ \hbox{where} \ H = \mathcal{H}_{m,n} = \lg x\rg\times\lg y\rg \cong C_{nm}\times C_m.
\end{array}
\end{equation}

\begin{lemma}\label{arc-tran-abelians}
Let $X=N_A(R(\mathcal{H}_{m,n}))$, where $A=\Aut(\G_{m,n,\ld})$.  Then$\,:$
\begin{enumerate}
\item[{\rm (a)}]\ if $n\leq 3$, then $X$ acts transitively on the $2$-arcs of $\,\G_{m,n,\ld},$ while
\item[{\rm (b)}]\ if $n>3$, then $X$ acts transitively on the arcs but not on the $2$-arcs of $\,\G_{m,n,\ld}$.
\end{enumerate}
Moreover, if $nm^2>4$ then $\,\G_{m,n,\ld}$ has girth $6$.
\end{lemma}

\begin{proof}
First, there exists an automorphism $\a$ of $\mathcal{H}_{m,n}$
that takes $(x,y)$ to $(x^{\ld-1}y, x^{-(\ld^2-\ld+1)}y^{-\ld})$.
To see this, note that  $x=(x^{\ld-1}y)^{-\ld}\cdot (x^{-(\ld^2-\ld+1)}y^{-\ld})^{-1}$,
so that $x^{\ld-1}y $ and $x^{-(\ld^2-\ld+1)}y^{-\ld}$ generate $\mathcal{H}_{m,n}$.
Next, $\ld\!-\!1\in\mz_n^{\,*}$ since $\ld(\ld\!-\!1)\equiv-1$ mod $n$, and so $x^{\ld-1}y$ has order $m_1n$ for some $m_1$ dividing $m$. It follows that $x^{(\ld-1)m_1n} = 1 = y^{m_1n}$, so that $mn$ divides $m_1n(\ld\!-\!1)$ and $m$ divides $m_1n$, and then $\frac{m}{m_1}$ divides $\ld\!-\!1$ and $n$, and hence divides ${\rm GCD}(\ld\!-\!1, n)=1$, so $m=m_1$. Thus $x^{\ld-1}y$ has order $mn$.
Similarly, $x^{-(\ld^2-\ld+1)}y^\ld$ has order $k$ for some $k$ dividing $m$ (since $\ld^2-\ld+1\equiv 0$ mod $n$),
and then because $y^{k\ld} = 1 = x^{k(\ld^2-\ld+1)}$,
we find that $\frac{m}{k}\ |\ \ld$ and $\frac{m}{k}\ |\ \frac{mn}{k}\ |\ \ld^2-\ld+1$,
and therefore $\frac{m}{k}=1$, which gives $m=k$, and thus $x^{-(\ld^2-\ld+1)}y^\ld$ has order $m$.

Note also that $\{1, x, x^\ld y\}^\a = \{1,x^{\ld-1}y,x^{\ld(\ld-1)}y^\ld x^{-(\ld^2-\ld+1)}y^{-\ld} \}
= \{1,x^{\ld-1}y,x^{-1}\} = x^{-1}\{1, x, x^\ld y\},$ so that $\a$ acts like left multiplication by $x^{-1}\!$
on the set $S= \{1, x, x^\ld y\}$.
It follows that $\s_{\a,x}$ is an automorphism of $\G_{m,n,\ld}$ that fixes $1_0$, takes $1_1$ to $x_1$,
and $x_1$ to $(xx^\a)_1=(x^\ld y)_1$, and $(x^\ld y)_1$ to $1_1$.
In particular, $\s_{\a,x}$ fixes the vertex $1_0$ and induces a $3$-cycle on its neighbours,
so $\G_{m,n,\ld}$ is locally arc-transitive.
Moreover, there exists an automorphism of $H$ that inverts every element, since $H$ is abelian,
and hence by Proposition~\ref{normal-arc-tran}, we find that $\G_{m,n,\ld}$ is arc-transitive.

Next, if $n\leq 3$, then since $n$ divides $\ld^2-\ld+1$, we have $n=1$ or $3$, and moreover, if $n=1$ then $\ld=0$, while if $n=3$  then $\ld=2$. In both cases it is easy to check that there is an automorphism $\b$ of $\mathcal{H}_{m,n}$ taking
$(x,y)$ to $(x^\ld y, x^{1-\ld^2}y^{-\ld})$.  This automorphism swaps $x$ with $x^\ld y$, and so by Theorem~\ref{normal-2-arc}, the group $X$ acts transitively on the $2$-arcs of $\G_{m,n,\ld}$.
Conversely, suppose $X$ acts transitively on the $2$-arcs of $\G_{m,n,\ld}$.  Then there exists $\b\in\Aut(\mathcal{H}_{m,n})$ such that $\b$ swaps $x$ with $x^\ld y$, so swaps $x^m$ with $x^{m\ld}$, and it follows that $\ld^2\equiv 1$ mod~$n$.
Then since $\ld^2-\ld+1\equiv 0$ mod $n$, we find that $\ld\equiv2$ mod $n$,
and so $0 \equiv \ld^2-\ld+1\equiv 4-2+1 \equiv 3$ mod $n$, which implies that $n \le 3$.

Finally, we consider the girth of $\G_{m,n,\ld}$, which is even, since $\G_{m,n,\ld}$ is bipartite.
In all cases, $\G_{m,n,\ld}$ contains a $6$-cycle, namely $(1_0, x_1, (x^{1-\ld}y^{-1})_0, (x^{1-\ld}y^{-1})_1, (x^{-\ld}y^{-1})_0, 1_1)$, and so its girth is at most $6$.  On the other hand, if the girth is at most $4$, then it is one of the two connected trivalent arc-transitive graphs of girth $4$, namely the complete bipartite graph $K_{3,3}$ or the cube graph $Q_3$ (see \cite{CN}).  These are the graphs that occur in the cases $(m,n,\ld)=(1,3,2)$ and $(2,1,0)$ respectively, and are also the only cases with order at most $8$, and hence with $nm^2 \le 4$.  See also \cite[Lemma~4.1]{Zhou-DA}. \end{proof}

\begin{xca}\label{tri-arc-tran-abelians}
Let $\,\G=\BiCay(H, \emptyset, \emptyset, S)$ be a connected trivalent normal edge-transitive bi-Cayley graph over an abelian group $H$. Then $\,\G\cong\G_{m,n,\ld}$ for some $m, n, \ld$.
\end{xca}

\begin{proof}
Let $X=N_{\Aut(\G)}(R(H))$. Then by Proposition~\ref{normal-arc-tran}, $\G(1_0)$ is an orbit of $X_{1_0}$, so there exists $\a\in\Aut(H)$ and $a\in H$ such that $\s_{\a, a}$ cyclically permutes the three neighbours of $1_0$ in $\G$, and it follows that $S=\{1, a, aa^\a\}$. Now let $b = aa^\a$. Then $\s_{\a, a}$ induces the $3$-cycle $(1_1,a_1,b_1)$ on $\G(1_0)$, so $ab^\a=1$,
which gives $b^\a = a^{-1}$. Hence in particular, $a$ and $b$ have the same order. Also by connectedness of $\G$, we have $H=\lg a, b\rg$. Next let $n = |\lg a\rg\cap\lg b\rg|$ and $m = |\lg a\rg : \lg a\rg\cap\lg b\rg| = |\lg b\rg : \lg a\rg\cap\lg b\rg|$.
Then we find that $|\lg a\rg| = |\lg b\rg| = nm$,  and $\lg a\rg\cap\lg b\rg=\lg a^m\rg=\lg b^m\rg$, and it follows that  $b^m=a^{\ld m}$ for some $\ld\in\mz_n^{\,*}$ when $n > 1$, or with $\ld = 0$ when $n = 1$.  Moreover, we have
$$a^{-m}=(b^\a)^m=(b^m)^\a=(a^{\ld m})^\a=(a^\a)^{\ld m}=(a^{-1}b)^{\ld m}=a^{-\ld m}b^{\ld m}=a^{-\ld m}a^{\ld^2 m}=a^{m(-\ld+\ld^2)},$$ and therefore $\ld^2-\ld+1\equiv0$ mod $n$ (because $a$ has order $nm$). Finally, letting $x=a$ and $y=a^{-\ld}b$, we have $H=\lg a, b\rg=\lg x\rg\times\lg y\rg$, with $S=\{1, a, b\}=\{1, x, x^{\ld}y\}$, and thus $\G\cong\G(m,n,\ld)$.
\end{proof}

\subsection{Trivalent edge-transitive graphs with small girth}
In this subsection, we use Proposition~\ref{abelian} to study trivalent edge-transitive graphs with girth at most $6$, and prove Theorem~\ref{th-girth}.
This is partially motivated by the work~in \cite{CN} and \cite{Kutnar-JCTB} on the classification of trivalent arc-transitive graphs of small girth. A natural question is whether there exists a trivalent semisymmetric graph of girth at most $6$. We will show that the answer is negative.
First, we prove the following:

\begin{lemma}\label{girth-number}
Let $\,\G$ be a connected trivalent edge-transitive graph of girth $\,6$, and let $A=\Aut(\G)$.
If $c$ is the number of $\,6$-cycles passing through an edge in $\,\G$, then $c=2, 4, 6, $ or $8$. Moreover,
\begin{enumerate}
  \item [{\rm (a)}]\ if $c=2$, then $A_v \cong C_3$ for every vertex $v$ of $\G$,
   or $A_v \cong S_3$ for every vertex $v$ of $\G$,
while
  \item [{\rm (b)}]\ if $c>2$, then $\,\G$ is isomorphic to 
  the Heawood graph, the Pappus graph, the generalised Petersen graph $P(8, 3)$, or the generalised Petersen graph
$P(10, 3)$, with $c = 8, 4, 6$ or $4$ respectively.
\end{enumerate}

\end{lemma}

\begin{proof}
Let $u$ be any vertex of $\G$.  Since every $6$-cycle passing through $u$ uses two of the
three edges incident with $u$, and every edge lies in $c$ $6$-cycles, the number of $6$-cycles through $u$ is $b = 3c/2$.
In particular, this is independent of $u$, and $c$ is even.
Also because $\G$ has valency $3$ and girth $6$, it is easy to see that there are at most eight $6$-cycles
passing through an edge of $\G$, and so $c=2, 4, 6$ or $8$ (and $b = 3, 6, 9$ or $12$).

Similarly, if $x$ is a vertex at distance $3$ from $u$, then there are at most three $6$-cycles passing
through both $u$ and $x$, and $|\G(x)\cap\G_2(u)|$ is at most $3$.
Moreover, if $|\G(x)\cap\G_2(u)|=3$, then we know from \cite[Lemma~4.6]{Zhou-Feng} that $\G$ is isomorphic
to the Heawood graph or the generalised Petersen graph $P(8, 3)$.  For these two graphs,
we have $c = 8$ and $6$ respectively, and from now on, we may suppose that $|\G(x)\cap\G_2(u)| \le 2$
for every vertex $x\in\G_3(u)$. In particular, since there are $2|\G_2(u)|=12$ edges between $\G_2(u)$ and $\G_{3}(u)$,
under the latter assumption we find that $b\leq 6$, and so $c\leq 4$.

Now suppose $c = 4$.  Then $b = 6$, and also $|\G_3(u)|=6$, with every vertex in $\G_3(u)$
adjacent to two of the vertices in $\G_2(u)$, so $|\G(x)\cap\G_2(u)| = 2$ for every $x\in\G_3(u)$.
Note that this holds for every vertex $u$.
Next, let $y\in\G_4(u)$. If we choose $v\in\G(u)$ such that $y\in\G_4(u)\cap\G_3(v)$,
so that $v$ is a neighbour of $u$ on some path of length $4$ from $u$ to $y$, then $|\G(y)\cap\G_2(v)| = 2$,
and in particular, $|\G(y)\cap\G_3(u)|\geq2$.
Then since each vertex in $\G_3(u)$ is adjacent to just one vertex in $\G_4(u)$,
while each vertex in $\G_4(u)$ is adjacent to two vertices in $\G_3(u)$, it follows that $|\G_4(u)| \le 3$.

Next, if $|\G(y)\cap\G_3(u)| = 3$ then $|\G_4(u)|=2$, with both vertices in $\G_4(u)$ having
three neighbours in $\G_3(u)$, so $\G$ has diameter $4$,
with $|V(\G)| = 1+|\G(u)|+|\G_2(u)|+|\G_3(u)|+|\G_4(u)| = 1+3+6+6+2 = 18$.
Then by what we know about edge-transitive trivalent graphs of small order from \cite{CD, Semi-symm-768},
we find that $\G$ is isomorphic to the Pappus graph.
On the other hand, if $|\G(y)\cap\G_3(u)| = 2$ for all $y \in \G_4(u)$, then there are at most three edges
from $\G_4(u)$ to $\G_5(u)$, so $|\G_5(u)|\leq 3$. But also if $z \in \G_5(u)$ then the same argument
as above shows that $|\G(z)\cap\G_4(u)|\geq2$, and it follows that $|\G_5(u)|=1$ and $|\G(z)\cap\G_4(u)| = 3$.
Hence in this case $\G$ has diameter $5$,
with $|V(\G)| = 1+|\G(u)|+|\G_2(u)|+|\G_3(u)|+|\G_4(u)|+|\G_5(u)| = 1+3+6+6+3+1 = 20$,
and then from \cite{CD, Semi-symm-768} we find that $\G$ is isomorphic to the generalised Petersen graph $P(10, 3)$.
(Note: the only other edge-transitive trivalent graph of order $20$ is the dodecahedral graph, which has girth $5$.)

Finally, suppose $c = 2$.  Then $b=3$, and by edge-transitivity, each of the three neighbours
of $u$ lies in two of the three $6$-cycles passing through $u$, and each $2$-arc of the form $(v,u,w)$
lies in exactly one of them. The same holds at any neighbour of $u$, and it follows that each of six vertices
in $\G_2(u)$ lies in exactly one of the three $6$-cycles through $u$.
Also just three of the vertices in $\G_3(u)$ lie on these cycles, and are then adjacent to two vertices
in $\G_2(u)$, while all other vertices in $\G_3(u)$ are adjacent to a single vertex in $\G_4(u)$.
Since there are $2|\G_2(u)|=12$ edges between $\G_2(u)$ and $\G_{3}(u)$,
we find that $|\G_3(u)|=6/2+6 = 9$, and the induced subgraph on $\{u\}\cup\G(u)\cup\G_2(u)\cup\G_3(u)$
is as shown in Figure~\ref{fig-1}. \\[-30pt]

\begin{figure}[ht]
\begin{center}
\unitlength 4mm
\begin{picture}(40,10)
{\footnotesize\put(20, 1){\circle*{0.4}}
\put(20,1){\line(0,1){1.5}}\put(20,1){\line(2,1){3}}\put(20,1){\line(-2,1){3}} \put(20,0){$u$}

\put(17, 2.5){\circle*{0.4}}\put(17,2.5){\line(-3,1){4}}\put(17,2.5){\line(-1,2){0.8}} 

\put(23, 2.5){\circle*{0.4}}\put(23,2.5){\line(3,1){4}}\put(23,2.5){\line(1,2){0.8}} 

\put(20, 2.5){\circle*{0.4}}\put(20,2.5){\line(-1,1){1.5}}\put(20,2.5){\line(1,1){1.5}}

\put(13, 4){\circle*{0.4}}
\put(13,4){\line(-1,1){3.5}}\put(9.5, 7.5){\circle{0.4}}
\put(13,4){\line(1,1){3.5}}\put(16.5, 7.5){\circle*{0.4}}

\put(16.2, 4){\circle*{0.4}}
\put(16.2,4){\line(1,1){3.5}}\put(19.7, 7.5){\circle*{0.3}}
\put(16.2,4){\line(2,1){3.8}}
\put(20, 5.8){\circle*{0.4}}

\put(27, 4){\circle*{0.4}}
\put(27,4){\line(1,1){3.5}}\put(30.5, 7.5){\circle*{0.3}}
\put(27,4){\line(-1,1){3.5}}\put(23.5, 7.5){\circle*{0.4}}

\put(23.8, 4){\circle*{0.4}}
\put(23.8,4){\line(-1,1){3.5}}\put(20.3, 7.5){\circle{0.4}} 
\put(23.8,4){\line(-2,1){3.8}}

\put(18.4, 4){\circle*{0.4}}
\put(18.4,4){\line(-1,1){3.5}}
\put(14.9, 7.5){\circle{0.4}}
\put(18.5,4){\line(-3,5){2}}

\put(21.5, 4){\circle*{0.4}}
\put(21.5,4){\line(1,1){3.5}}
\put(25, 7.5){\circle*{0.3}}
\put(21.4,4){\line(3,5){2}}}
 \end{picture}
\end{center}\vspace{-.5cm}
\caption{Local subgraph of $\G$ in the case of $c=2$} \label{fig-1}
\end{figure}
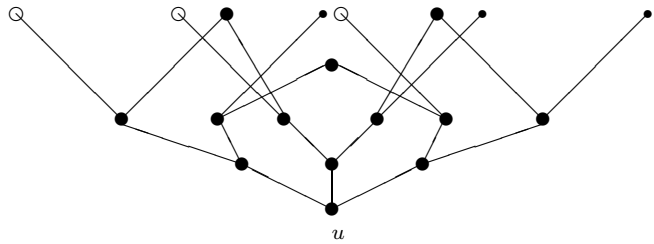

Now from Figure~\ref{fig-1} (or the argument leading to it) we see that
any automorphism in $A=\Aut(\G)$ that fixes $u$ and each of its three neighbours must fix all the vertices of each
of the three $6$-cycles passing through $u$, and hence fixes every vertex of $\G_2(u)$.
Hence $A_{(u\cup\G(u))}=A_{(u\cup\G(u)\cup\G_2(u))}$ for every $u \in V(\G)$.
By connectedness and induction, it follows that $A_{(u\cup\G(u))}$ fixes every vertex of $\G$,
and is therefore trivial. In particular, $A_u$ acts faithfully on $\G(u)$, and so by edge-transitivity,
$A_u \cong C_3$ or $S_3$.
Moreover, if $w$ is any neighbour of $u$, then $|A_u| = 3|A_{uw}| = |A_w|$,
and thus either $A_v \cong C_3$ for all $v \in V(\G)$, or $A_v \cong S_3$ for all $v \in V(\G)$.
\end{proof}


The above lemma helps us to find all trivalent edge-transitive graphs of girth $6$, as follows.

\begin{xca}\label{th-girth6}
Let $\,\G$ be a connected trivalent edge-transitive graph of grith $6$.
Then either $\,\G\cong\G_{m,n,\ld}$ with $nm^2>9$ $($as defined in $(\ref{abelian groups}))$,
or $\,\G$ is isomorphic to the Heawood graph, the Pappus graph, the generalised Petersen graph $P(8, 3)$,
or the generalised Petersen graph $P(10, 3)$.
In particular, in all cases, the graph $\,\G$ is arc-transitive.
\end{xca}

\begin{proof}
Let $A=\Aut(\G)$.  First we show that $\G$ is bipartite.  If $\G$ is arc-transitive, then this follows
from \cite[Corollary~6.3]{Feng-Nedela} or from what was proved for girth $6$ in \cite{CN},
while if $\G$ is not arc-transitive, then $\G$ is semisymmetric and hence bipartite.
Moreover, $A = \Aut(\G)$ acts transitively on each part of $\G$.

Next, let $\{u, v\}\in E(\G)$, and take $B=\lg A_u, A_v\rg$. Then $B$ is edge- but not vertex-transitive on $\G$,
and by edge-transitivity, we have $|A_u: A_{uv}|=3= |A_v:A_{uv}|$.
If there are more than two $6$-cycles passing through $\{u, v\}$ in $\G$, then by Lemma~\ref{girth-number}(b),
we know that $\G$ is isomorphic to the Heawood graph, the Pappus graph, $P(8, 3)$ or $P(10, 3)$,
all of which are arc-transitive.

{}From now on, we will suppose that  there are exactly two $6$-cycles passing through $\{u, v\}$,
and hence (by Lemma~\ref{girth-number}(a)) that $A_u\cong A_v\cong C_3$ or $S_3$.

Under this assumption, it follows that if $\G$ is arc-transitive, then $A$ is a quotient of one of the Djokovi{\' c}-Miller
amalgams $1'$ or $2'$ and $2''$ from \cite{DjokovicMiller} used in \cite{CD} and \cite{CN},
while if $\G$ is semisymmetric then $A$ is a quotient of one of the Goldschmidt amalgams $G_1$
and $G_{1}^{\ 1}$ from \cite{Goldschmidt} used in \cite{Semi-symm-768}.
In particular, $A$ can be obtained from one of those amalgams by adding extra relations
to their defining presentations, to force a circuit of length $6$ in the resulting graph.

When $\G$ is arc-transitive, we find more specifically from \cite{CN} that for girth $6$ the amalgam $2''$
can be eliminated, and furthermore, that if $A_u\cong A_v\cong C_3$
then $A$ is a quotient of the ordinary $(2,3,6)$ triangle group $\langle\, a,h \ | \  a^2 = h^3 = (ha)^6 = 1 \,\rangle$,
with the image of $h$ and generating $A_u$, and the image of $h^a$ generating $A_v$,
while if $A_u\cong A_v\cong S_3$ then $A$ is a quotient of either the extended $(2,3,6)$ triangle
group $\langle\, a,h,p \ | \  a^2 = h^3 = (ha)^6 = p^2 = (ap)^2 = (hp)^2 = 1 \,\rangle$, with the images
of $h$ and $p$ and generating $A_u$, and the images of $h^a$ and $p^a$ ($=p$) generating $A_v$.
In both cases, the elements $h$ and $k = h^a$ satisfy the relations $h^3 = k^3 = (hk)^3 = 1$,
and their images generate a subgroup of $A$ (of index $2$ or $4$) that acts transitively on each part of $\G$.

On the other hand, when $\G$ is semisymmetric, we can perform the same kind of analysis as carried out
in \cite{CN} for the case of girth $6$.

If $A_u\cong A_v\cong C_3$ then $A$ is a quotient of the free product $C_3 * C_3 = \langle\, h,k \ | \  h^3 = k^3 = 1 \,\rangle$,
with the images of $h$ and $k$ generating $A_u$ and $A_v$, and having girth $6$ implies that some further
relation $w = 1$ is satisfied, where $w = w(h,k)$ is a word of length $6$ in the generators $h$ and~$k$.
Without loss of generality, $w = h^{r_1}k^{s_1}h^{r_2}k^{s_2}h^{r_3}k^{s_3}$ with $r_i = \pm 1$
and $s_i = \pm 1$ for $1 \le i \le 3$, and then an easy computation using {\sc Magma}~\cite{BCP} shows that
every such relation forces the quotient to have order at most $24$, except in the cases where
$w = (hk)^3$, $(hk^{-1})^3$, $(h^{-1}k)^3$ or $(h^{-1}k^{-1})^3$.
But we know from \cite{Semi-symm-768} that a semisymmetric trivalent graph has order at least $54$
and hence at least $81$ edges, so $|A| \not\le 24$.  Also we can replace each of $h$ and $k$ by its inverse,
and so we may conclude that $w = (hk)^3$, and again we have elements $h$ and $k$ satisfying the
relations $h^3 = k^3 = (hk)^3 = 1$.

Similarly, if $A_u\cong A_v\cong S_3$ then $A$ is a quotient
of $\langle\, h,k,p \ | \  h^3 = k^3 = p^2 = (hp)^2 = (kp)^2 = 1 \,\rangle$,
with $A_u$ and $A_v$ being the images of $\langle h,p\rangle$ and  $\langle k,p\rangle$,
and in this case girth $6$ implies that some relation of the form $w = 1$ or $w = p$ is satisfied,
where $w$ is as above.   Here the analogous {\sc Magma} computation shows that there are only four
such relations that produce a quotient of order greater than $60$, namely $(hk)^3 = 1$ and the others
obtainable by replacing $h$ and/or $k$ by their inverses.

Hence in all cases, whether $\G$ is arc-transitive or semisymmetric, $A =\Aut(\G)$ contains two elements
$h$ and $k$ that fix the vertices $u$ and $v$ and induce $3$-cycles on $\G(u)$ and $\G(v)$, respectively,
and satisfy the relations $h^3 = k^3 = (hk)^3 = 1$.
Also it is clear that the subgroup $L$ generated by $h$ and $k$ is edge-transitive, with two orbits on vertices
of $\G$ (namely the two parts of $\G$), and has index at most $2$ in $B=\lg A_u, A_v\rg$;
indeed $B = \langle h,k \rangle$ or $\langle h,k,p \rangle$ in each of the above cases.

Now let $J$ be the subgroup of $L = \langle h,k \rangle$ generated by $x = hkh$ and $y = hk^{-1}$.
Then $J$ is normal in~$L$, because
$$x^h = kh^2 = kh^{-1} = y^{-1} \ \hbox{ and } \ y^h = k^{-1}h = k^{2}h = y^{-1}x,$$
while
$$x^k = k^{-1}hkhk = k^{-1}(hk)^2 = k^{-1}(hk)^{-1} =  k^{-2}h^{-1} = kh^{-1} = y^{-1}
\ \hbox{ and } \ y^k = k^{-1}h = k^{2}h = y^{-1}x,$$
and it follows that also $J$ is abelian, since $y^x = y^{hkh} = (y^{-1}x)^{kh} = (x^{-1}yy^{-1})^h = (x^{-1})^h = y$.

Moreover, $Jh = Jk$, since $hk^{-1} = y \in J$, and therefore $L/J = \lg Jh, Jk \rg =\lg Jh\rg$,
and then since $h$ has order $3$ it follows that $|L/J| = 1$ or $3$.
On the other hand, if $L = J$ then $L$ is abelian so $h$ commutes with $k$ and therefore the
edge-transitive group $L = \langle h,k \rangle$ has order at most $9$, which is impossible
since $\G$ has girth $6$.  Hence $|L:J| = 3$.
In particular, $h \not\in J$ and $k \not\in J$,
so $J$ is complementary to each of $L_u = \lg h \rg$ and $L_v = \lg k \rg$ in $L$,
and it follows that $J$ acts semi-regularly on $V(\G)$, with two orbits, namely the two parts of $\G$.
Thus  $\G$ is a bi-Cayley graph over the abelian group $J$.

But furthermore, we can show that $J$ is normal in $A=\Aut(\G)$, in all cases.
For if $\G$ is semisymmetric and $A_u\cong A_v\cong C_3$,
then $A = \lg h,k \rg = L$, while if $\G$ is semisymmetric and $A_u\cong A_v\cong S_3$,
then $A = \lg h,k,p \rg$ for some $p$ satisfying $p^2 = (hp)^2 = (kp)^2 = 1$,
and then
$$x^p = (hkh)^p = h^{-1}k^{-1}h^{-1} = x^{-1} \ \hbox{ and } \ y^p = (hk^{-1})^p = h^{-1}k = h^{-1}k^{-2} = x^{-1}y.$$
On the other hand, if $\G$ is arc-transitive and $A_u\cong A_v\cong C_3$,
then $A = \lg h,a \rg$ for some involution $a$ conjugating $h$ to $k$,
and then
$$x^a = (hkh)^a = khk = h^{-1}k^{-1}h^{-1} = x^{-1} \  \hbox{ and } \  y^a = (hk^{-1})^a = kh^{-1} = y^{-1},$$
while if $\G$ is arc-transitive and $A_u\cong A_v\cong S_3$,
then $A = \lg h,a,p \rg$ for some $a$ as above, and some $p$ satisfying $p^2 = (ap)^2 = (hp)^2 = 1$,
and then $h^p = h^{-1}$ while $k^p = (aha)^p = ah^{-1}a = k^{-1}$, and so again $x^p = x^{-1}$ and $y^p = x^{-1}y.$
Thus  $\G$ is a normal bi-Cayley graph over the abelian group $J$.

We can now apply Proposition~\ref{tri-arc-tran-abelians},
which tells us that $\G$ is isomorphic to $\G_{m,n,\ld}$ for some $m, n, \ld$ with $\ld^2-\ld+1 \equiv 0$ mod $n$.
Hence in particular, $\G$ is arc-transitive.

Finally, if $nm^2\leq 9$ then it is easy to see that $(n,m)=(3,1), (7,1), (1, 2),$ or $ (1,3)$.
Also if $(n,m)=(3,1)$ or $(1,2)$, then $\G_{m,n,\ld}\cong\K_{3,3}$ or $Q_3$, each of which has girth $4$,
while if $(n,m)=(7,1)$ or $(1,3)$, then $\G_{m,n,\ld}$ is isomorphic to the Heawood graph or Pappus graph,
which do not satisfy our assumption on the number of $6$-cycles through an edge.
Thus $nm^2>9$, completing the proof.
\end{proof}

We can now prove Theorem~\ref{th-girth}, showing that all trivalent edge-transitive graphs of girth
at most 6 are known, and are arc-transitive.
\medskip

\f{\bf Proof of Theorem~\ref{th-girth}}.\ 
The arc-transitive trivalent graphs of girth less than 6 are known (see \cite{CN} or \cite{Kutnar-JCTB}),
and by \cite[Lemma~4.1]{Zhou-DA}, there is no semisymmetric trivalent graph of girth less than 6,
and Proposition~\ref{th-girth6} gives all trivalent edge-transitive graphs of girth exactly 6,
with none being semisymmetric.
\hfill$\Box$ \smallskip

\section{Edge-transitive bi-dihedrants}
\label{sec:ET-bidihedrants}

In this section, we investigate edge-transitive bi-dihedrants (that is, bi-Cayley graphs over dihedral groups).
We will show there are no semisymmetric bi-dihedrants of valency at most $5$, and on the other hand,
by considering normal edge-transitive bi-dihedrants, that there exist semisymmetric bi-dihedrants of valency $2k$
for every odd integer $k > 1$.  Also we give a characterisation of $6$-valent edge-regular semisymmetric bi-Cayley graphs over
a dihedral group $D_{n}$ of odd degree $n$.  In turn, this enables us to answer two questions proposed
in 2001 by Maru\v si\v c and Poto\v cnik on semisymmetric tetracirculants \cite{MP-gf}.

\subsection{The smallest valency of semisymmetric bi-dihedrants}
The aim of this subsection is to prove the following theorem.

\begin{theorem}\label{th-valency5}
Let $\G={\rm BiCay}(H,R,L,S)$ be a connected semisymmetric bi-Cayley graph over a dihedral group
$H=\lg\, a, b\ |\ a^n = b^2 = (ab)^2 = 1 \,\rg \cong D_n\ ($for some $n\geq 3)$.
Then the valency of $\,\G$ is at least $6$.
\end{theorem}

\begin{proof}
First, $\G$ is bipartite, and its two parts are the orbits of both $A=\Aut(\G)$ and its subgroup
of $H$ on $V(\G)$. It follows that $R=L=\emptyset$.
By Proposition~\ref{n}(b), we may assume that $1\in S$, and since $\G$ is not vertex-transitive,
we find by Proposition~\ref{normaliser} that there is no automorphism of $H$ mapping $S$ to $S^{-1}$.

It follows that $S \setminus \{1\}$ cannot consist entirely of involutions, and so $S$ must contain
at least one element of order greater than $2$.
On the other hand, as $\G$ is connected, $S$ must contains at least one element of the form
$ba^i$ (with $i\in\mz_n$), and then by replacing $b$ by $ba^i$ if necessary, we may suppose that  $i = 0$
and hence that $S$ contains $b$.  If, however, this is the only involution in $S$, then all
other elements of $S$ are powers of $s$, and so the automorphism of $H$ taking $(a,b)$ to $(a^{-1},b)$
inverts every element of $S$, so takes $S$ to $S^{-1}$, which is impossible.
Similarly, if $S$ contains just one other involution of the form $ba^j$ (with $j\in\mz_n$),
then the automorphism of $H$ taking $(a,b)$ to $(a^{-1},ba^j)$ swaps the involutions $b$ and $ba^j$
and inverts every other element of $S$, and so takes $S$ to $S^{-1}$, which again is impossible.
Hence $S$ contains at least three involutions, as well as an element of order greater than $2$.
In particular, the valency $|S|$ of $\G$ is at least $5$.

To complete the proof, we need only show that $|S|$ cannot be $5$.  So we assume the contrary.
Then we know that $S=\{1, b, ba^i, ba^j, a^k\}$ where $0 < i < j < n$, and $a^k$ is not an involution.
In particular, we cannot have $\,i \equiv -i \equiv k$ mod $n,\,$ or $\,j \equiv -j \equiv k$ mod $n,\,$
or $\,i-j \equiv j-i \equiv k$ mod $n. \,$
Also the fact that no automorphism of $H$ taking $(a,b)$ to $(a^{-1},b)$ or $(a^{-1},ba^{2i})$
or $(a^{-1},ba^{2j})$ is allowed to take $S$ to $S^{-1}$ implies that $\,i \not\equiv -j$ mod $n,\,$
and $\,j \not\equiv 2i$ mod $n,\,$ and $i \not\equiv 2j$ mod $n$.

We proceed by considering the number of cycles of length $4$ through a given edge,
which by edge-transitivity of $\G$ must be a constant.

Up to reversal, there are either three or six $4$-cycles through the edge $\{1_0,1_1\}$,
namely the three of the form $(1_0,1_1,x_0,x_1)$ for $x \in \{b,ba^i,ba^j\}$, plus \\[-15pt]

\begin{center}\begin{tabular}{ll}
$(1_0,1_1,(ba^i)_0,(a^k)_1)$, $\,(1_0,1_1,(ba^i)_0,b_1)\,$ and $\,(1_0,1_1,(a^{-k})_0,b_1)$
 & \ if \ $k \equiv i$ mod $n$, \\[+4pt]
or \ $(1_0,1_1,(ba^j)_0,(a^k)_1)$, $\,(1_0,1_1,(ba^j)_0,b_1)\,$ and $\,(1_0,1_1,(a^{-k})_0,b_1)$
 & \ if \ $k \equiv j$ mod $n$, \\[+4pt]
or \ $(1_0,1_1,b_0,(ba^i)_1)$, $\,(1_0,1_1,b_0,(a^k)_1)\,$ and $\,(1_0,1_1,(a^{-k})_0,(ba^i)_1)$
 & \ if \ $k \equiv -i$ mod $n$, \\[+4pt]
or \ $(1_0,1_1,b_0,(ba^j)_1)$, $\,(1_0,1_1,b_0,(a^k)_1)\,$ and $\,(1_0,1_1,(a^{-k})_0,(ba^j)_1)$
 & \ if \ $k \equiv -j$ mod $n$, \\[+4pt]
or \ $(1_0,1_1,(ba^j)_0,(ba^i)_1)$, $\,(1_0,1_1,(ba^j)_0,(a^k)_1)\,$ and $\,(1_0,1_1,(a^{-k})_0,(ba^i)_1)$
 & \ if \ $k \equiv j-i$ mod $n$, \\[+4pt]
or \ $(1_0,1_1,(ba^i)_0,(ba^j)_1)$, $\,(1_0,1_1,(ba^i)_0,(a^k)_1)\,$ and $\,(1_0,1_1,(a^{-k})_0,(ba^j)_1)$
 & \ if \ $k \equiv i-j$ mod $n$. \\[-5pt]
\end{tabular}\end{center}

\noindent
Note that no two of the above six congruences involving $k$ can occur simultaneously,
by the restrictions we have on $i$, $j$ and $k$.
Hence up to reversal, the number of $4$-cycles through any given edge is $3$ or $6$.
\smallskip

Next, up to reversal the $4$-cycles through $\{1_0,b_1\}$ are $(1_0,b_1,b_0,1_1)$
and $(1_0,b_1,(ba^k)_0,(a^k)_1)$, plus \\[-15pt]

\begin{center}\begin{tabular}{ll}
$(1_0,b_1,(ba^i)_0,(ba^i)_1)$, $\,(1_0,b_1,(ba^i)_0,1_1)\,$ and $\,(1_0,b_1,(a^{-k})_0,1_1)$
 & \ if \ $k \equiv i$ mod $n$, \\[+4pt]
or \ $(1_0,b_1,(ba^j)_0,(ba^j)_1)$, $\,(1_0,b_1,(ba^j)_0,1_1)\,$ and $\,(1_0,b_1,(a^{-k})_0,1_1)$
 & \ if \ $k \equiv j$ mod $n$, \\[+4pt]
or \ $(1_0,b_1,b_0,(ba^i)_1)$, $\,(1_0,b_1,b_0,(a^k)_1)\,$ and $\,(1_0,b_1,(a^k)_0,(a^k)_1)$
 & \ if \ $k \equiv -i$ mod $n$, \\[+4pt]
or \ $(1_0,b_1,b_0,(ba^j)_1)$, $\,(1_0,b_1,b_0,(a^k)_1)\,$ and $\,(1_0,b_1,(a^k)_0,(a^k)_1)$
 & \ if \ $k \equiv -j$ mod $n$, \\[+4pt]
or \ $(1_0,b_1,(a^i)_0,(ba^i)_1)$
 & \ if \ $2i \equiv 0$ mod $n$, \\[+4pt]
or \ $(1_0,b_1,(a^j)_0,(ba^j)_1)$
 & \ if \ $2j \equiv 0$ mod $n$.
\end{tabular}\end{center}

It follows that the number of $4$-cycles through $\{1_0,b_1\}$ is not $3$ or $6$,
unless $\,0 \equiv 2i$ or $2j$ mod $n\,$ and $\,k \not\equiv \pm(i-j)$ mod $n$.

But now suppose $2i \equiv 0$ mod $n$ and $k \not\equiv \pm(i-j)$ mod $n$.
Then $n$ is even, and $i \equiv \frac{n}{2}$ mod $n$, and up to reversal the $4$-cycles
through $\{1_0,(ba^j)_1\}$ are $(1_0,(ba^j)_1,(ba^j)_0,1_1)$
and $(1_0,(ba^j)_1,(ba^{j+k})_0,(a^k)_1)$, plus \\[-12pt]

\begin{center}\begin{tabular}{ll}
$(1_0,(ba^j)_1,(ba^j)_0,1_1)$, $\,(1_0,(ba^j)_1,(ba^j)_0,(a^k)_1)\,$ and $\,(1_0,(ba^j)_1,(a^k)_0,(a^k)_1)$
 & \ if \ $k \equiv j$ mod $n$, \\[+4pt]
or \ $(1_0,(ba^j)_1,b_0,1_1)$, $\,(1_0,(ba^j)_1,b_0,b_1)\,$ and $\,(1_0,(ba^j)_1,(a^{-k})_0,1_1)$
 & \ if \ $k \equiv -j$ mod $n$. \\[-5pt]  
\end{tabular}\end{center}

\noindent
Hence the number of $4$-cycles through the edge $\{1_0,(ba^j)_1\}$ is $2$ or $5$, contradiction.
The same holds when the roles of $i$ and $j$ are reversed, and so this completes the proof.
\end{proof}

\subsection{A class of normal edge-transitive bi-dihedrants}
In this subsection, we construct a class of normal edge-transitive bi-Cayley graphs over dihedral groups
of degree $5$ or more, and thereby prove there exists a semisymmetric bi-dihedrant of valency $2k$
for every odd integer $k \ge 3$. \medskip

\begin{example}\label{exam1}
{\rm Let $n$ and $k$ be integers with $n \ge 5$ and $k \ge 2$, such that
there exists an element $\ld$ of order $2k$ in $\mz_n^{\,*}$ such that
$$1 + \ld^2 + \ld^4 + \ldots \ld^{2(k-2)}+\ld^{2(k-1)} \equiv 0 \ \hbox{ mod } n.$$
Now let $H$ the dihedral group $D_{n}=\lg\, a, b\ |\ a^n = b^2 = (ab)^2 = 1 \,\rg$ of degree $n$,
and for each $i \in \mz_k$, let
$$c_i = 1+\ld^2+\ld^4+ \dots +\ld^{2(i-1)}+\ld^{2i}\ \ \hbox{ and } \ \ d_i = \ld c_i = \ld+\ld^3+\ld^5+\dots+\ld^{2i-1}+\ld^{2i+1},$$
and then define $\G(n,\ld, 2k)$ as the $2k$-valent bi-Cayley graph ${\rm BiCay}(H, \emptyset, \emptyset,S)$ over $H,$
where 
$$S = S(n, \ld, 2k) = \{a^{c_i} : i \in \mz_k \}\, \cup \,\{ba^{d_i} :  i \in \mz_k \}.$$
It is easy to see that $\G(n,\ld, 2k)$ contains
the $2n$-cycles $(1_0,a_1,a_0,(a^2)_1,(a^2)_0, \ldots, (a^{n-1})_1,(a^{n-1})_0,1_1)$
and $(b_0,b_1,(ba)_0,(ba)_1,(ba^2)_0,(ba^2)_1, \ldots, (ba^{n-1})_0,(ba^{n-1})_1)$,
and the edge $(1_0,b_1)$, so $\G(n,\ld, 2k)$ is connected.

Also it is easy to see that $|S| = 2k$, and $\,c_{k-1} \equiv d_{k-1} \equiv 0$ mod $n,\,$
and $\,1+\ld d_i \equiv c_{i+1}$ mod $n\,$ for all $i \in \mz_k$.
Next let $\a$ be the automorphism of $H$ that takes $(a,b)$ to $(a^\ld, ba).$
Then $S^\a=bS$, and 
$\s_{\a, b}$ is an automorphism of $\G(n, \ld, 2k)$ that fixes the vertex $1_0$ and cyclically permutes
the $2k$ neighbours of $1_0$; indeed $\s_{\a, b}$ takes $(a^{c_i})_1$ to $(ba^{\ld c_i})_1 = (ba^{d_i})_1$,
and $(ba^{d_i})_1$ to $(b^{2}a^{1+\ld d_i})_1 = (a^{c_{i+1}})_1,$ for all $i \in \mz_k$.
Hence in particular, this shows that $\G(n, \ld, 2k)$ is normal edge-transitive.}
\end{example}

The following natural problem arises.\medskip

\f{\bf Problem~A}\
{\em Determine which of the graphs $\,\G(n, \ld, 2k)$ are semisymmetric.}\medskip

We will give some partial answers to this problem, in the situation where $\ld^k\equiv-1$ mod $n$.
\smallskip


\begin{proposition}\label{arc-tran-dihe}
If $k$ is even and $\ld^k\equiv-1$ mod $n$, then $\,\G(n, \ld, 2k)$ is arc-transitive.
\end{proposition}

\begin{proof}
Let $\b$ be the automorphism of $H$ taking $(a,b)$ to $(a^{-1},ba^\ell)$
where $\ell = d_{(k-2)/2} = \ld+\ld^3+\ldots+\ld^{k-1}$.
Clearly $\b$ has order $2$.
Also $(a^{c_i})^\b = a^{-c_i}\in S^{-1}$ for all $i \in \mz_k$,
while $(ba^{d_i})^\b = ba^{\ell-d_i}$ for all $i \in \mz_k$,
and because  $\ld^k \equiv -1$ mod $n\,$ we find that if $\,0 < 2i+1 < k\,$ then
\\[+6pt]
${}$\hskip 0.6cm
\begin{tabular}{ll}
$\ell-d_i \hskip -4pt$ & $\equiv \  (\ld+\ld^3+\ldots+\ld^{k-1}) - ( \ld+\ld^3+\ldots+\ld^{2i-1}) $ \\
 & $\equiv \  \ld^{2i+1}+\ld^{2i+3}+\ldots+\ld^{k-1} $ \\
 & $\equiv \  \ld^{2i+1}+\ld^{2i+3}+\ldots+\ld^{k-1}
 + (\ld+\ld^{k+1}) + (\ld^3+\ld^{k+3}) + \ldots + (\ld^{2i-1}+ \ld^{k+2i-1}) $ \\
 & $\equiv \  \ld+\ld^3+ \ldots +\ld^{2i-1}+\ld^{2i+1}+\ld^{2i+3}+\ldots+\ld^{k-1}+\ld^{k+1}+\ld^{k+3}+\ldots+\ld^{k+2i-1} $ \\
 & $\equiv \  d_{(k+2i)/2}\,$ mod $n$,\\[+6pt]
\end{tabular}
\par\noindent
while if $\,2i+1 > k\,$ then
\\[+6pt]
${}$\hskip 0.6cm
\begin{tabular}{ll}
$\ell-d_i \hskip -4pt$ & $\equiv \  (\ld+\ld^3+\ldots+\ld^{k-1}) - ( \ld+\ld^3+\ldots+\ld^{2i-1}) $ \\
 & $\equiv \  -\ld^{k+1}-\ld^{k+3}+\ldots-\ld^{2i-1}
 \ \equiv \  \ld+\ld^3+\ldots+\ld^{2i-k-1}
 \ \equiv \  d_{(2i-k)/2}\,$ mod $n$, \\[+6pt]
\end{tabular}
\par\noindent
and so $(ba^{d_i})^\b = ba^{\ell-d_i} = ba^{d_{(2i\pm k)/2}}  = (ba^{d_{(2i\pm k)/2}})^{-1} \in S^{-1}$ for all $i \in \mz_k$.
Hence the automorphism $\b$ takes the set $S = S(n, \ld, 2k)$ to $S^{-1}$, and so by Proposition~\ref{normaliser},
we find that $\G(n, \ld, 2k)$ is vertex-transitive, and therefore arc-transitive.\end{proof}

\begin{proposition}\label{semi-dihe}
If $k$ is odd and $\ld^k\equiv-1$ mod $n$, then $\,\G(n, \ld, 2k)$ is semisymmetric.
\end{proposition}

\begin{proof}
First, let $\ell = d_{(k-3)/2} = \ld+\ld^3+\ldots+\ld^{k-2}$. Then since $k$ is odd and $\ld^k \equiv -1$ mod $n$, we have
$$0 \ \equiv \ 1+\ld^2+\dots+\ld^{2(k-1)} 
\ \equiv \ c_{(k-1)/2}+\ld^{k}d_{(k-3)/2}
\ \equiv \ c_{(k-1)/2}-d_{(k-3)/2} \ \hbox{ mod } n,
$$
so $\,1+\ld\ell \,\equiv\, c_{(k-1)/2} \,\equiv\,  d_{(k-3)/2} \,\equiv\, \ell\,$ mod $n$,
and therefore  $\,(ba^{\ell})_0^{\s_{\a, b}} = ((ba^{\ell})^\a)_0 = (ba^{1+\ld\ell})_0 = (ba^{\ell})_0.\,$
Hence $(ba^{\ell})_0$ is fixed by $\s_{\a, b}$, which cyclically permutes the $2k$ neighbours of $1_0$,
and it follows that those $2k$ neighbours of $1_0$ are also the $2k$ neighbours of $(ba^{\ell})_0$.

Before proceeding, we note that since $1+\ld \equiv \ell$ mod $n$, we also have \\[+3pt]
$2\ell \, \equiv \,  1+\ell+\ld\ell
\ \equiv \ 1+(\ld+\ld^3+\ldots+\ld^{k-2})+(\ld^2+\ld^4+\ldots+\ld^{k-1})
\ \equiv \ 1+\ld+\ld^2+\ld^3+\ldots+\ld^{k-2}+\ld^{k-1}, $ \\[+3pt]
and therefore \
$2\ell(1-\ld) \, \equiv \, (1+\ld+\ld^2+\ld^3+\ldots+\ld^{k-2}+\ld^{k-1})(1-\ld)  \ \equiv \ 1-\ld^k  \ \equiv \  2\,$ mod $n$. \\[-10pt]

Next, let $B$ be the set of all vertices of $\G$ having the same neighbourhood as $1_0$.
Note that $B$ contains both $1_0$ and $(ba^{\ell})_0$,  and  therefore $|B| \ge 2$.
We claim that $B$ is a block of imprimitivity for $\Aut(\G)$ on~$V(\G)$.
To see this, note that if $\s\in\Aut(\G)$ then all vertices of $B^\s$ must have the same neighbourhood
(since the same holds for vertices in $B$), and hence if $B\cap B^\s\neq \emptyset$,
then every vertex in $B^\s$ has the same neighbourhood as $1_0$, so $B^\s \subseteq B$ and this gives $B^\s=B$.
Hence in particular, $B$ is a block of imprimitivity for $R(H) \le \Aut(\G)$ on $H_0$,
and so $H_B = \{\, x \in H \ | \ x_0 \in B \,\}$ is a subgroup of $H$, with order $|H_B| = |B|$
because $R(H)$ acts regularly on~$H_0$.
But also $B \subseteq \G(1_1)$, and so $B$ is a block of imprimitivity for $\Aut(\G)_{1_1}$ on $\G(1_1)$ as well.
Hence $|B|$ divides $|\G_{1}(1_1)|=2k$. 
Moreover, $B$ contains $(ba^{\ell})_0$, so $H_B$ contains the involution $ba^{\ell}$,
and hence $H_B$ is dihedral, of order $|H_B| = |B| = 2j$ for some $j$ dividing $k$.

Now suppose that $\G$ is vertex-transitive.
Then some automorphism $\theta$ of $\G$ takes $1_0$ to $1_1$, and it follows that $C = B^{\theta}$
is the set of all vertices of $\G$ having the same neighbourhood as $1_1$,
and $C$ is a  a block of imprimitivity for $\Aut(\G)_{1_0}$ on $\G(1_0),$
and the subgroup $H_C = \{\, y \in H \ | \ y_1 \in C \,\}$ of $H$ is dihedral, of order $|C| = |B| = 2j$.
In particular, since the automorphism $\s_{\a, b}$ fixes $1_0$ and cyclically permutes
the $2k$ neighbours of $1_0$, the block $C$ is preserved by $\s_{\a, b}^{\ 2k/|C|} = \s_{\a, b}^{\ k/j},$
and hence also by $\s_{\a, b}^{\ k}$.  Accordingly, $C$ contains the image of $1_1$ under $\s_{\a, b}^{\ k}$,
namely $(ba^{(k-3)/2}))_1 = (ba^{\ell})_1$, and therefore $H_C$ contains $ba^{\ell}$.

On the other hand, consider the automorphism $\tau = \s_{\a, b}R(b)$.
This takes $h_1$ to $(bh^{\a}b)_1$ for all $h \in H$, and so its effect on $H_1$ is the same as the
permutation induced by the automorphism $\psi$ of $H$ that takes $a$ to $\,ba^{\a}b = ba^{\ld}b = a^{-\ld}\,$
and $b$ to $\,bb^{\a}b = b(ba)b =  ba^{-1}.\,$
In particular, $\tau$ fixes $1_1$, and so $\tau$ preserves $C$ (the set of all vertices
of $\G$ having the same neighbourhood as $1_1$), and $\psi$ preserves $H_C$.
It follows that $H_C$ contains $(ba^{\ell})^{\psi}  = ba^{-1-\ld\ell} = ba^{-\ell}$,
and hence also $(ba^{-\ell})^{-1}ba^{\ell} = a^{2\ell}$,
and hence also $(a^{2\ell})^{1-\ld} = a^{2\ell(1-\ld)} = a^2$, because $(1-\ld)2\ell \equiv 2$ mod $n$.
But $H_C$ has order $2j$, which divides $2k$ and hence divides $|\mz^{\,*}| = \phi(n)$,
and so $|H_C \cap \lg a \rg| = j \le k \le \phi(n)/2 < n/2$. Thus $a^2$ cannot lie in $H_C$,
and which is a contradiction.
\end{proof}

\f{\bf Remarks}: \ 
We believe that the hypothesis  $\ld^k \equiv -1$ mod $n$ in Proposition \ref{semi-dihe}
is not actually required, and that Proposition \ref{semi-dihe} can be extended to all cases other
than those covered by Proposition \ref{arc-tran-dihe}.  In other words, we believe that
$\G(n, \ld, 2k)$ is arc-transitive if and only if $k$ is even and $\ld^k \equiv -1$ mod $n$.
Furthermore, we believe that if $\ld^k \not\equiv -1$ mod $n$, then the graph $\G = \G(n, \ld, 2k)$ is
not just semisymmetric, but edge-regular (or equivalently, the stabiliser in $\Aut(\G)$ of any edge is trivial).
This certainly holds in all cases where $n \le 300$, such as $(n, \ld, k) = (21,2,3)$, $(68,9,4)$ or $(35,2,6)$,
as shown by a computation using {\sc Magma}~\cite{BCP}.
In the next subsection, we will prove it holds whenever $k = 3$.

\subsection{The case {\em k} = 3}
By Theorem~\ref{th-valency5}, every semisymmetric bi-dihedrant has valency at least $6$,
and therefore $3$ is the smallest possible value of $k$ of interest in this section.
Also by Proposition~\ref{semi-dihe}, we know that when $k=3$ the graph $\G(n, \ld, 6)$ is semisymmetric
if $\ld^3\equiv-1$ mod $n$.
In this subsection, we prove that $\G(n, \ld, 6)$ is edge-regular (and therefore semisymmetric)
whenever $\ld^3\not\equiv-1$ mod $n$, and this gives a complete solution for Problem~A in the case $k=3$.




\begin{theorem}\label{edge-regular-graphs}
The graph $\,\G(n, \ld, 2k)$ is semisymmetric whenever $k = 3$,
and moreover, if $\,k = 3$ and $\ld^3\not\equiv-1$ mod $n$, then $\,\G(n, \ld, 2k)$ is edge-regular,
with cyclic vertex-stabiliser.
\end{theorem}

\begin{proof}
Let $\G=\G(n, \ld, 2k)$ and $A=\Aut(\G)$.
We know from Proposition~\ref{semi-dihe} that $\G$ is semisymmetric whenever $\ld^3\equiv-1$ mod $n$,
and hence in what follows, we will assume that $\ld^3\not\equiv-1$ mod $n$.

The smallest value of $n$ for which this happens is $21$, with $\ld = \pm 2$ or $\pm 10$ (in $\mz_{21}$),
and the next smallest $n$ is $39$, with $\ld = \pm 4$ or $\pm 10$ (in $\mz_{39}$).
Note that $n$ must be odd, for otherwise $\ld$ would be odd but then $c_{k-1} = c_2 = 1 + \ld^2 + \ld^4$
could not be $0$ mod $n$.

By considering $4$- and $6$-cycles that contain the edge $\{1_0,1_1\}$, we will prove that
the stabiliser $A_{1_{0}1_{1}}$ of the arc $(1_{0},1_{1})$ is trivial, and then that the stabiliser $A_{\{1_{0},1_{1}\}}$
of the edge $\{1_{0},1_{1}\}$ is trivial, so that $\G$ is semisymmetric, and edge-regular.
\medskip

But first, we will set some notation for later use. Define
\\[+4pt]
$\, H_{0c}=\{ (a^i)_0 : \, i \in \mz_n \}$, \ $\, H_{0d}=\{ (ba^i)_0 : \, i \in \mz_n \}$, \
$\, H_{1c}=\{ (a^i)_1 : \, i \in \mz_n \}\,$ \  and \ $\, H_{1d}=\{ (ba^i)_1 : \, i \in \mz_n \}$.
\\[+4pt]
These form a partition of $V(\G)$ into four subsets of size $n$,
with $H_0 = H_{0c }\cup H_{0d}$ and $H_1 = H_{1c }\cup H_{1d}$,
and they are blocks of imprimitivity for $R(H)\rtimes\lg\s_{\a, b}\rg$ on $V(\G)$,
with $\s_{\a, b}$ preserving each of $H_{0c}$ and $H_{0d}$, and interchanging $H_{1c}$ with $H_{1d}$.
In fact both $R(a)$ and $\s_{\a, b}^{\ 2}$ preserve these four subsets, while all of $R(b)$, $\s_{\a, b}$
and $\s_{\a, b}R(b)$ do not, and it follows that the kernel of the action of $R(H)\rtimes\lg\s_{\a, b}\rg$
on $\{H_{0c}, H_{0d}, H_{1c}, H_{1d}\}$ is the index $4$ subgroup $M = \lg R(a), \s_{\a, b}^{\ 2} \rg$.
\smallskip

Also define
\\[+4pt]
$\, \Delta_0 = \{(a^{-c_i})_0 : \, i \in \mz_3\}$, \ $\, \Phi_0 = \{(ba^{d_i})_0 : \, i \in \mz_3\}$, \
$\, \Delta_1 = \{(a^{c_i})_1 : \, i \in \mz_3\}\,$ \ and \ $\, \Phi_1 = \{(ba^{d_i})_1 : \, i \in \mz_3\}$,
\\[+4pt]
so that $\Delta_0 \cup \Phi_0 = \ G(1_1)$ and $\Delta_1 \cup \Phi_1 = \ G(1_0)$, the neighbourhoods
of the vertices $1_1$  and $1_0$. \smallskip

Next, the assumption that $\ld^3\not\equiv-1$ mod $n$ implies that no power
of $\ld$ is congruent to $-1$ mod $n$,
and so the set $\{-\ld^j : 1 \le j \le 5 \}$ is disjoint from $\{\ld^j : 1 \le j \le 5 \}$.
It follows that the set $\G_2(1_0)$ of vertices at distance $2$ from $1_0$ is the union of
the following three disjoint sets:
\\[+5pt]
${}$\hskip 0.4cm
\begin{tabular}{ll}
$U_1 \hskip -4pt$ & $= \  \{ (a^{c_i-c_j})_0 : \, i,j \in \mz_3, \, i \ne j \,\} \
  = \ \{ (a^\ell)_0 :  \, \ell \in \{ \pm 1, \pm \ld^2, \pm \ld^4 \} \},$ \\[+3pt]
$U_2 \hskip -4pt$ & $= \  \{ (ba^{c_i+d_j})_0 : \, i,j \in \mz_3 \} \
  = \ \{ (ba^\ell)_0 : \, \ell \in \{ 0, 1, \ld, 1+\ld, -\ld^4,  -\ld^5, 1-\ld^5, \ld-\ld^4, -\ld^4-\ld^5 \} \},$ \\[+3pt]
$U_3 \hskip -4pt$ & $= \  \{ (a^{d_i-d_j})_0 : \, i,j \in \mz_3, \, i \ne j \,\} \
  = \ \{ (a^\ell)_0 :  \, \ell \in \{ \pm \ld, \pm \ld^3, \pm \ld^5 \} \}.$ \\[+6pt]
\end{tabular}

Note that each of the vertices in $U_1$ and $U_3$ has only one common neighbour with $1_0$,
while each of the vertices in $U_2$ has two common neighbours with $1_0$,
and hence up to reversal there are only nine $4$-cycles containing $1_{0}$,
namely $(1_0,(a^{c_i})_1,(ba^{c_i+d_j})_0, (ba^{d_j})_1)$ for each pair $(i,j) \in \mz_3 \times \mz_3$.
It follows that the stabiliser $A_{1_0}$ of the vertex $1_{0}$ must preserve the set $U_2$.
Moreover, the arc-stabiliser $A_{1_{0}1_{1}}$ must permute the three edges
$\{(ba^{d_j})_0, (ba^{d_j})_1\}$ among themselves, as these are the only edges between
vertices of $\G(1_1)$ and $\G(1_0)$ that lie in $4$-cycles through $1_0$ and $1_1$.

Therefore $A_{1_{0}1_{1}}$ preserves the subset $\Phi_1$ of $\G(1_0)$
and its complement $\Delta_1 \setminus \{1_1\} = \{(a^{c_i})_1: \, i \in \{0,1\}\}$ in $\G(1_0)\setminus \{1_1\}$,
as well as the subset $\Phi_0$ of $\G(1_1)$ and its complement $\Delta_0 \setminus \{1_0\} = \{(a^{-c_i})_0: \, i \in \{0,1\}\}$
in $\G(1_1)\setminus \{1_0\}$, and it also sets up a pairing between the subsets $\Phi_0$ and $\Phi_1$.
\smallskip

The situation is illustrated in Figure \ref{fig-3}.

\begin{figure}[ht]
\begin{center}
\unitlength 4mm
\begin{picture}(30,20.5)
{\footnotesize\put(2, 10){\circle*{0.4}}\put(0.7,9.8){$1_0$}
\put(2,10){\line(1,1){6}} \put(2,10){\line(3,2){6}}\put(2,10){\line(3,1){6}}
\put(2,10){\line(1,-1){6}} \put(2,10){\line(3,-2){6}}\put(2,10){\line(3,-1){6}}

\put(7.6,16.7){$1_1$}\put(8, 16){\circle*{0.4}} \put(8,16){\line(5,2){10}}\bezier{600}(8, 16)(9,16)(18,19)
 \put(8,16){\line(6,-1){11}}\bezier{600}(8, 16)(7,16)(19,13)\bezier{600}(8, 16)(6,16)(19,12)

\put(7.6,14.6){$a_1$}\put(8, 14){\circle*{0.4}} \put(8,14){\line(5,2){10}}\bezier{600}(8, 14)(9,14)(18,17)
 \bezier{600}(8, 14)(7,14)(19,11)\bezier{600}(8, 14)(7,14)(19,10)\bezier{600}(8, 14)(6,14)(19,9)

\put(6.1,10.4){$(a^{1+\ld^2})_1$}\put(8, 12){\circle*{0.4}}\put(8, 12){\circle*{0.4}} \put(8,12){\line(5,2){10}}\bezier{600}(8, 12)(9,12)(18,15)
 \bezier{600}(8, 12)(7,12)(19,8)\bezier{600}(8, 12)(7,12)(19,7)\bezier{600}(8, 12)(6,12)(19,6)

\put(7.8,7){$b_1$}\put(8, 8){\circle*{0.4}}\put(8,8){\line(5,-2){10}}\bezier{600}(8, 8)(9,8)(18,5)
\bezier{600}(8, 8)(9,8)(19,14)\bezier{600}(8, 8)(9,8)(19,11)\bezier{600}(8, 8)(9,8)(19,8)

\put(7.2,5.0){$(ba^\ld)_1$}\put(8, 6){\circle*{0.4}}\put(8,6){\line(5,-2){10}}\bezier{600}(8, 6)(8,6)(18,3)
\bezier{600}(8, 6)(8,6)(19,13)\bezier{600}(8, 6)(8,6)(19,10)\bezier{600}(8,6)(8,6)(19,7)

\put(6,2.7){$(ba^{\ld+\ld^3})_1$}\put(8, 4){\circle*{0.4}}\put(8,4){\line(5,-2){10}}\bezier{600}(8, 4)(8,4)(18,1)
\bezier{600}(8, 4)(8,4)(19,12)\bezier{600}(8, 4)(8,4)(19,9)\bezier{600}(8,4)(8,4)(19,6)

\put(18, 20){\circle*{0.4}}\put(18.5,20.2){$(a^{-1})_0$}
\put(18, 19){\circle*{0.4}}\put(18.5,19){$(a^{-1-\ld^2})_0$}

\put(18, 18){\circle*{0.4}}\put(18.5,18){$a_0$}
\put(18, 17){\circle*{0.4}}\put(18.5,17){$(a^{-\ld^2})_0$}

\put(18, 16){\circle*{0.4}}\put(18.5,16){$(a^{1+\ld^2})_0$}
\put(18, 15){\circle*{0.4}}\put(18.5,15){$(a^{\ld^2})_0$}

\put(18, 0){\circle*{0.4}}\put(18.5,0){$(a^{\ld^3})_0$}
\put(18, 1){\circle*{0.4}}\put(18.5,1){$(a^{\ld+\ld^3})_0$}

\put(18, 2){\circle*{0.4}}\put(18.5,2){$(a^{-\ld^3})_0$}
\put(18, 3){\circle*{0.4}}\put(18.5,3){$(a^{\ld})_0$}

\put(18, 4){\circle*{0.4}}\put(18.5,4){$(a^{-\ld-\ld^3})_0$}
\put(18, 5){\circle*{0.4}}\put(18.5,5){$(a^{-\ld})_0$}

\put(19, 6){\circle*{0.4}}\put(19.5,6){$(ba^{1+\ld+\ld^2})_0$}
\put(19, 7){\circle*{0.4}}\put(19.5,7){$(ba^{1+\ld^2})_0$}
\put(19, 8){\circle*{0.4}}\put(19.5,8){$(ba^{-\ld^4-\ld^5})_0$}
\put(19, 9){\circle*{0.4}}\put(19.5,9){$(ba)_0$}
\put(19, 10){\circle*{0.4}}\put(19.5,10){$(ba^{1+\ld+\ld^3})_0$}
\put(19, 11){\circle*{0.4}}\put(19.5,11){$(ba^{1+\ld})_0$}
\put(19, 12){\circle*{0.4}}\put(19.5,12){$(ba^{\ld+\ld^3})_0$}
\put(19, 13){\circle*{0.4}}\put(19.5,13){$(ba^{\ld})_0$}
\put(19, 14){\circle*{0.4}}\put(19.5,14){$b_0$}

}
 \end{picture}
\end{center}
\caption{The ball of radius $2$ centered at the vertex $1_0$ in $\G(n,\ld,6)$} \label{fig-3}
\end{figure}
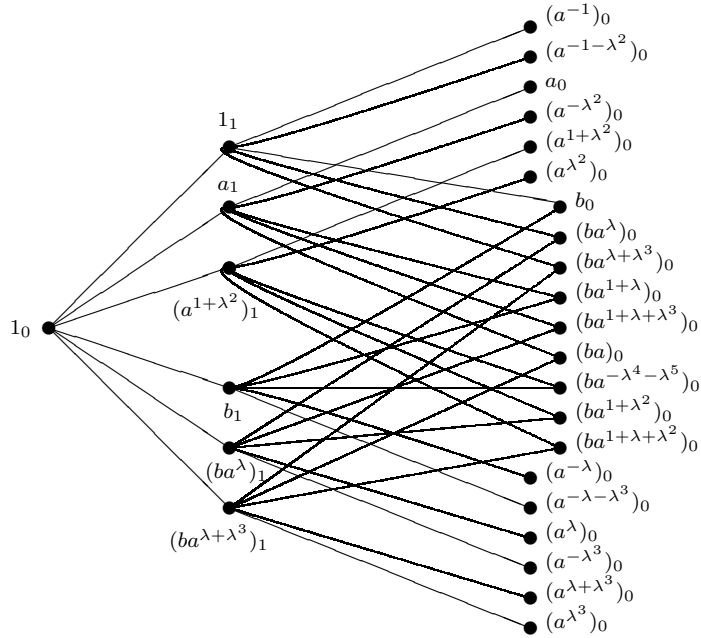

On the other hand, the vertex $1_0$ lies in many $6$-cycles, but
it is not difficult to check that if $n \ge 39$ then a $3$-arc of the form $(u_1,1_0,1_1,w_0)$
with $u_1 = (a^{c_i})_1 \in \Delta_1 \setminus \{1_1\}$ and $w_0 = (a^{-c_j})_0 \in \Delta_0 \setminus \{1_0\}$
lies in no $6$-cycle when $i = j$,
and in just one $6$-cycle $((a^{c_i})_1,1_0,1_1,(a^{-c_j})_0, (a^{c_i-c_j})_1, (a^{c_i-c_j})_0)$ when $i \ne j$.
Also in the smallest cases, where $n = 21$, this $3$-arc lies in three $6$-cycles when $i = j$,
but only one $6$-cycle when $i \ne j$.
It follows that the arc-stabiliser $A_{1_{0}1_{1}}$ must preserve the set of two edges of the form
$\{(a^{c_i})_0, (a^{c_i})_1\}$ for $i \in \{0,1\}$, and this sets up a pairing between the subsets
$\Delta_0 \setminus \{1_0\}$ and $\Delta_1 \setminus \{1_1\}$.

Thus we have a pairing between the five vertices of $\G(1_1) \setminus \{1_0\}$
and the five vertices of $\G(1_0) \setminus \{1_1\}$, given by $(ba^{d_i})_0 \leftrightarrow (ba^{d_i})_1$
for $i \in \mz_3$ and $(a^{-c_i})_0 \leftrightarrow (a^{c_i})_1$ for $i \in \{0,1\}$,
such that $A_{1_{0}1_{1}}$ permutes the corresponding edges among themselves.
By edge-transitivity, a similar thing holds for every edge in $\G$.

\medskip
It now follows that $A_{1_0}$ acts faithfully on $\G(1_0)$.
For suppose $g$ is an automorphism in $A_{1_0}$ that fixes every neighbour of $1_0$ in $\G$.
Then $g$ fixes $1_1$, and also fixes the partner in $\G(1_1)$ of every other vertex in $\G(1_0)$,
and so fixes all of $\G(1_1)$.  By edge-transitivity, the same argument applies to other neighbours
of $1_0$, and it follows that $g$ fixes every vertex at distance $2$ from $1_0$.
Then by induction and connectedness, we find that $g$ fixes all vertices of $\G$, and hence $g$ is trivial.

Also the action of $A_{1_0}$ is imprimitive on $\G(1_0)$,
with two blocks $\Delta_1$ and $\Phi_1 = \Delta_1^{\,\s_{\a,b}}$ of size $3$,
because if $u_1 \in \Delta_1 \cap \Delta_1^{\,g}$ for some $g \in A_{1_0}$,
then $g$ must preserve the set of all vertices of $\G(1_0)$ that lie in $4$-cycles containing
the edge $\{1_0,u_1\}$, namely $\Phi_1 = \{(ba^{d_i})_1 : \, i \in \mz_3\}$,  and hence also $\Delta_1^{\,g} = \Delta_1$.
Thus $A_{1_0}$ is isomorphic to a subgroup of the wreath product $S_3 \wr C_2$,
so $A_{1_0}$ is a $\{2,3\}$-group, of order dividing $72$.  \smallskip

Next, let $A^*$ be the subgroup of $A$ preserving the parts $H_0$ and $H_1$ of $\G$,
so that $A^*$ has index $1$ or $2$ in $A$.  Then $A^*$ contains $R(H)$ and $A_{1_0}$,
and since $R(H)$ acts regularly on each part of $\G$, it follows that $A^*$ is the
complementary product $R(H)A_{1_0}$ of these two subgroups.
If $p$ is any prime divisor of $|A^*|$ such that $p > 3$, then $p$ cannot divide $|A_{1_0}|$
and therefore $p$ divides $|R(H)| = |H| = 2n$, so $p$ divides $n$, and then since $\ld^4+\ld^2+1\equiv 0$ mod $n$,
we find that $p \ne 5$, so $p\ge 7$.  Moreover, for every such $p$, the dihedral subgroup $R(H)$ of order $2n$
has a unique Sylow $p$-subgroup $P$, which is then also a Sylow $p$-subgroup of $A^*$
and is normal in the subgroup $R(H)\rtimes\lg\s_{\a, b}\rg$ of order $12n$ in $A^*$.
Hence the index of its normaliser in $A^*$ divides $|A_{1_0}: \lg \s_{\a, b}\rg|$, which divides $72/6 = 12$,
and because $p \ge 7$, it follows that $P$ is normal in $A^*$.
The product of all such Sylow subgroups is therefore a (cyclic) normal Hall $\{2,3\}'$-subgroup $N$ of $A^*$.

\medskip
We can use this fact to prove that $A_{1_0}$ preserves the set $H_{0c}$ (of all vertices of the form $(a^i)_0$),
and that $A_{1_{0}1_{1}}$ preserves the set $H_{1c}$ (of all vertices of the form $(a^j)_1$).
Before doing that, observe that $\s_{\a, b}$ preserves $H_{0c}$, and that $A_{1_0} = \lg A_{1_{0}1_{1}},\s_{\a, b} \rg$,
since $\s_{\a, b}$ fixes $1_0$ and acts regularly on $\G(1_0)$.
Hence all we need to do is prove that $A_{1_{0}1_{1}}$ preserves both $H_{0c}$ and $H_{1c}$.

If $n$ is coprime to $3$, then the normal subgroup $N$ of $A^*$ must be the cyclic subgroup generated by $R(a)$,
and the four sets above are its orbits, which are therefore blocks of imprimitivity for $A^*$  on $V(\G)$,
and so $A_{1_01_1}$ preserves both $H_{0c}$ and $H_{1c}$.

On the other hand, suppose $n \equiv 0$ mod $3$.  Then $n \not\equiv 0$ mod $9$ (since $\ld^4+\ld^2+1\equiv 0$ mod $n$),
and it follows that $n=3|N|$, and $N$ is the cyclic subgroup generated by $R(a)^3$,
and  $H_{0c} = 1_0^N \cup a_0^N \cup (a^{-1})_0^N$.
Now~if~$x \in A_{1_01_1}$, then $(1_0^N)^x = (1_0^{\, x})^N =  1_0^N \subseteq H_{0c}$,
and  by our earlier observations, $x$ preserves the set $\Delta_1 \setminus \{1_1\} = \{(a^{c_0})_1,(a^{c_1})_1\}$,
and hence also preserves the set $U_1 =  \{ (a^\ell)_0 :  \, \ell \in \{ \pm 1, \pm \ld^2, \pm \ld^4 \} \}$ of certain vertices
at distance $2$ from $1_0$.  It follows that
$$
(a_0^N \cup (a^{-1})_0^N)^x \ = \ a_0^{Nx} \cup (a^{-1})_0^{Nx}
\ = \ a_0^{xN} \cup (a^{-1})_0^{xN} \ \subseteq \ U_1^N \ \subseteq \ H_{0c},
$$
and thus $A_{1_01_1}$ preserves $H_{0c}$.
Similarly, every $x \in A_{1_01_1}$ must preserve the set $\Phi_1 = \{(ba^{d_i})_1 : \, i \in \mz_3\}$,
and hence also preserve the set $U_3 =  \{ (a^\ell)_0 :  \, \ell \in \{  \pm \ld, \pm \ld^3, \pm \ld^5  \} \}$,
and then since $\ld $ is a unit mod $n$ we have $\ld \equiv \pm 1$ mod $3$,
and so $a_1^N\cup (a^{-1})_1^N = (a^{\ld})_1^N\cup (a^{-\ld})_1^N$.
It follows that
$$
(a_1^N\cup (a^{-1})_1^N)^x \ = \ (a^{\ld})_1^{Nx}\cup (a^{-\ld})_1^{Nx}
\ = \ (a^{\ld})_1^{xN}\cup (a^{-\ld})_1^{xN} \ \subseteq \ U_3^N \ \subseteq \ H_{1c},
$$
and then since also $(1_1^N)^x = (1_1^{\, x})^N =  1_1^N \subseteq H_{1c}$,
we find that $A_{1_01_1}$ preserves $H_{1c}$, as required.

\medskip
Now let $X$ be the subgraph of $\G$ induced on all vertices in $H_{0c} \cup H_{1c}$
and let $Y$ be the subgraph induced on all vertices in $H_{0c} \cup H_{1d}$.
Then each of $X$ and $Y$ is a trivalent bi-Cayley graph over the cyclic group $\lg a\rg \cong C_n$;
indeed clearly $X$ is the graph $\BiCay(\lg a\rg, \emptyset,\emptyset, \{a^{c_i} : \, i \in \mz_3\})$,
while $Y$ is isomorphic to $\BiCay(\lg a\rg, \emptyset,\emptyset, \{a^{d_i} : \, i \in \mz_3\})$.
Also the subgroup of $A$ generated by $R(a)$ and $\s_{\a, b}^{\, 2}$ acts transitively on the edges
of both subgraphs.  Hence each is a connected edge-transitive bi-Cayley graph over the
cyclic group $\lg a\rg \cong C_n$, and so by Proposition~\ref{abelian} each of them is arc-transitive,
and then since $n \ge 21$, it follows from Corollary~\ref{cyclic} that each of $X$ and $Y$ is $1$-arc-regular.

We saw above that the arc-stabiliser $A_{1_{0}1_{1}}$ preserves $H_{0c}$ and $H_{1c}$,
and it follows that $A_{1_{0}1_{1}}$ preserves $H_{0d}$ and $H_{1d}$ as well.
Hence $A_{1_{0}1_{1}}$ induces a group of automorphisms of each of $X$ and $Y$.
Then since $X$ is arc-regular, and $A_{1_{0}1_{1}}$ fixes the arc $(1_{0},1_{1})$ of $X$, we find
that $A_{1_{0}1_{1}}$ acts trivially on $X$.  In particular, $A_{1_{0}1_{1}}$ fixes all the vertices $(a^i)_0$
that are common to $X$ and $Y$, including the vertices in $\Delta_1 = \{(a^{c_i})_1: \, i \in \mz_3\}$,
and it follows that $A_{1_{0}1_{1}}$ fixes all the vertices in $\Phi_1 = \{(ba^{d_i})_1: \, i \in \mz_3\}$,
each of which lies in a unique $4$-cycle with $1_0$ and a given vertex of $\Delta_1$.
(Indeed $A_{1_{0}1_{1}}$ acts trivially on $Y$.)
Hence the subgroup $A_{1_{0}1_{1}}$ of $A_{1_0}$ acts trivially on $\Delta_1 \cup \Phi_1 = \G(1_0)$,
and so $A_{1_{0}1_{1}}$ itself is trivial.

In particular, every automorphism of $\G$ is uniquely determined by its effect on $1_0$ and $1_1$.
Also we now find that $A_{1_0} = \lg A_{1_{0}1_{1}},\s_{\a, b} \rg = \lg \s_{\a, b} \rg$, which has order $6$,
so $A^* = R(H)A_{1_0} = R(H) \rtimes \lg \s_{\a, b} \rg$, which has order $12n$.
Consequently the index $4$ normal subgroup $M = \lg R(a), \s_{\a, b}^{\ 2} \rg$ of $A^*$ has odd order $3n$,
and is therefore a characteristic subgroup of of $A^*$.
Moreover, since $R(a)^{\s_{\a, b}} = \s_{\a, b}^{-1}R(a)\s_{\a, b}$ takes $1_0$ to $(a^\ld)_0$,
we know that $R(a)^{\s_{\a, b}} = R(a^{\ld}) = R(a)^{\ld}$, and so conjugation by $\s_{\a, b}^{\, 2}$
induces an automorphism of $\lg R(a) \rg$ of order $3$, centralising the $3$-part of $\lg R(a) \rg$
because this has order at most $3$.  It follows that $J = \lg R(a) \rg$ is the only cyclic normal
subgroup of order $n$ in $M$, and $J$ is therefore characteristic in both $M$ and $A^*.$

\medskip
We complete the proof by using these facts to show that the edge-stabiliser $A_{\{1_{0}1_{1}\}}$ is trivial.
If that is not the case (so that $\G$ is arc-transitive), then the latter subgroup has order $2$,
and is generated by an automorphism $\theta$ that swaps the vertices $1_0$ and $1_1$.
Moreover, from our earlier observations, we know that $\theta$ must preserve the subsets
$\Delta_1 \setminus \{1_1\} = \{(a^{c_i})_1: \, i \in \{0,1\}\}$ and $\Delta_0 \setminus \{1_0\} = \{(a^{-c_i})_0: \, i \in \{0,1\}\},$
and so $\theta$ swaps $a_1$ with $(a^{-c})_0$
where $c = c_0 \equiv 1$ or $c = c_1 \equiv 1+\ld^2$ mod $n$.  It then follows that also $\theta$ swaps $a_0$ with $(a^c)_1$.
Similarly, $\theta$ swaps $\Phi_0 = \{ba^{d_i})_0 : \,i \in \mz_3 \}$ with $\Phi_1 = \{ba^{d_i})_1 : \,i \in \mz_3 \}$,
in a way that preserves the pairing $(ba^{d_i})_0 \leftrightarrow (ba^{d_i})_1$ for $i \in \mz_3$.
In particular, $\theta$ swaps $b_0$ with $(ba^d)_1$ for some $d \in \{d_0,d_1,d_2\}$, and then
also swaps $b_1$ with $(ba^d)_0$.

Now consider what happens to the automorphisms $R(a)$, $R(b)$ and $\s_{\a, b}$ under
conjugation by $\theta$.  Clearly $\theta$ normalises $A^*$ hence normalises
the characteristic subgroups $J$ and $M$ of $A^*$.  Also the fact that $A_{1_{0}1_{1}}$ is trivial
implies that every automorphism is uniquely determined by its effect on $1_0$ and $1_1$.
First $R(b)^{\theta}$ takes $1_0$ to $(ba^d)_0$, and $1_1$ to $(ba^d)_1$,
and therefore $R(b)^{\theta} = R(ba^d) = R(b)R(a)^d$.
Next, $R(a)^{\theta}$ lies in $J^{\theta} = J = \lg R(a) \rg$,
and as $R(a)^{\theta}$ takes $1_0$ to $(a_1)^{\theta} = (a^{-c})_0$, we have $R(a)^{\theta} = R(a)^{-c}$.
Similarly, $\s_{\a, b}^{\,\theta}$ fixes $1_1$ and takes $1_0$ to $(b_1)^{\theta} = (ba^d)_0$,
and it follows that $\s_{\a, b}^{\,\theta} = (\s_{\a, b}R(b))^{\, j}$ for some $j$,
because $\s_{\a, b}R(b)$ fixes $1_1$ and induces the $6$-cycle
$(1_0,b_0,(a^{-1})_0,(ba^{\ld})_0,(a^{-(1+\ld^2)})_0,(ba^{\ld+\ld^3})_0)$ on $\G(1_1)$.
In fact $j = \pm 1$, since $\s_{\a, b}$ and $\s_{\a, b}R(b)$ have order $6$,
and then 
it follows that 
$d = d_2 = 0$ or $d = d_1 = \ld+\ld^3$.
Hence the automorphism of $A^* = R(H) \rtimes \lg \s_{\a, b} \rg$ induced by $\theta$
takes $(R(a),R(b),\s_{\a, b})$ to $(R(a)^{-c},R(b)R(a)^d,(\s_{\a, b}R(b))^{\,\pm 1})$.

In particular, $\theta$ conjugates the involution $\s_{\a, b}^{\, 3}$ to $(\s_{\a, b}R(b))^{\, 3},$
and hence $\theta$ swaps $(ba^{\ld})_0$ with $(ba^{\ld})_1$.
Also $R(a) = R(a)^{\theta^2} = R(a)^{(-c)^2}$ and so $c^2 \equiv (-c^2) \equiv 1$ mod $n$,
and this eliminates the possibility that $c = c_1 = 1+\ld^2$,
because $(1+\ld^2)^2 \equiv 1+2\ld^2+\ld^4 \equiv \ld^2 \not\equiv 1$ mod $n$.
Thus $c = 1$, and so $R(a)^{\theta} = R(a)^{-1}$.
Finally, since $R(ba^\ld)^{\theta}$ takes $1_0$ to $((ba^\ld)_1)^{\theta} = (ba^\ld)_0$,
we find that $\theta$ centralises $R(ba^\ld)$, and so
$R(ba^\ld) = R(ba^\ld)^{\theta} = (R(b)R(a)^\ld))^{\theta} = R(b)R(a)^{d-\ld}$;
which gives $2\ld \equiv d \equiv \ld$ or $\ld+\ld^3$ mod $n$.
Both of these cases are impossible, however, since $\ld \not\equiv 0$ and $\ld^2 \not\equiv 1$ mod $n$.

\smallskip
Hence no such $\theta$ exists, and therefore $\G$ is semisymmetric and edge-regular.
\end{proof}

The graphs $G(n, \ld, k)$ with $k = 6$ investigated above provide the answers to two open questions.

For non-empty subsets $S_{00}, S_{01}, S_{10}$ and $S_{11}$ of $\mz_n$,
define $X=\mathcal{T}(S_{00}, S_{01}, S_{10}, S_{11})$ as the graph with
vertex set $\mz_n\times\mz_2\times\mz_2,$ and edges all pairs of the form
$\{(x, 0, i), (y, 1, j)\}$ where $i, j\in\mz_2$ and $y-x\in S_{ij}$.
The translation mapping $t \mapsto t+1$ on $\mz_n$ clearly induces a semi-regular
automorphism of $X$ of order $n$, with four cycles on $V(X)$.
Any graph admitting a semi-regular automorphism $\pi$ with four cycles on vertices
is called an {\em $(n, \pi)$-tetracirculant}.
Such graphs were considered in a 2001 paper \cite{MP-gf} by Maru\v si\v c and Poto\v cnik,
who showed that every semisymmetric $(n, \pi)$-tetracirculant is isomorphic some
$\mathcal{T}(S_{00}, S_{01}, S_{10}, S_{11})$.
Also if $S_{00}=S_{01}=R$ and $S_{10}=S_{11}=T$, then $\mathcal{T}(R, R, T, T)$
is called a {\em generalised Folkman tetracirculant}.
(See \cite{MP-gf} for the definition of generalised Folkman graph.)

In their 2001 paper, Maru\v si\v c and Poto\v cnik~\cite{MP-gf} posed the two questions below.
\medskip

\f{\bf Problem~B}\ {\rm \cite[Problem~4.3]{MP-gf}}\label{mp-1} \
{\em Is there a semisymmetric tetracirculant which is not a generalised Folkman tetracirculant\,}?
\medskip

\f{\bf Problem~C}\ {\rm \cite[Problem~4.9]{MP-gf}}\label{mp-2} \
{\em Is there a semisymmetric $(n, \pi)$-tetracirculant $\,\G$ such that the four orbits of
$\lg\pi\rg$ are blocks of imprimitivity of $\Aut(\G)$, but $\,\G$ is not a generalised Folkman tetracirculant\,}?
\medskip

We can now give the answer ``Yes'' to both questions.

For suppose $X=\mathcal{T}(R, R, T, T)$ is a semisymmetric generalised Folkman tetracirculant.
Then the vertices $(0, 0, 0)$ and $(0, 0, 1)$ have exactly the same neighbours in $X$,
namely all the vertices $(y, 1, j)$ with $j \in \mz_2$ and $y \in R \cup T$,
and so there exists an automorphism of $X$ that swaps $(0, 0, 0)$ with $(0, 0, 1)$
and fixes all others.  In particular, this implies that $X$ can not be edge-regular.

But now every bi-dihedrant $\BiCay(D_n,\emptyset,\emptyset,S)$ is a tetracirculant, admitting the natural cyclic
subgroup $R(C_n)$ as a group of automorphism with four orbits.
In particular, every bi-dihedrant $\G = \G(n, \ld, 6)$ considered above with $\ld^3\not\equiv-1$ mod $n$
is a semisymmetric $(n, \pi)$-tetracirculant, and furthermore,
the four orbits $H_{0c}$, $H_{0d}$, $H_{1c}$, $H_{1d}$ of the semi-regular cyclic subgroup $\lg R(a) \rg$
of its automorphism group $\Aut(\G)$ are blocks of imprimitivity for $R(H) \rtimes \lg \s_{\a, b} \rg = A^* = \Aut(\G)$.
On the other hand, by Theorem~\ref{edge-regular-graphs}, every such $\G$ is edge-regular,
and so $\G$ cannot be a semisymmetric generalised Folkman tetracirculant.

\subsection{Proof of Theorem~\ref{th-bi-dihedrant}}
Combining Theorem~\ref{th-valency5}, Proposition~\ref{semi-dihe} and Theorem~\ref{edge-regular-graphs}
together gives a proof of Theorem~\ref{th-bi-dihedrant}.  Note that the graphs in  Theorem~\ref{edge-regular-graphs}
are worthy, because every unworthy graph is not edge-regular, by the same argument as given in the penultimate
paragraph of the previous subsection.

\subsection{Edge-regular bi-dihedrants of valency $6$}
Here we give a classification of all $6$-valent edge-regular semisymmetric bi-Cayley graphs over a dihedral group
$D_n$ of odd degree $n$. We begin as follows:

\begin{lemma}\label{edge-regular-normal}
If $\,\G$ is a connected $6$-valent bi-Cayley graph over the dihedral group $H = D_n$ of order~$2n$,
where $n$ is odd, and $\,\G$ is both semisymmetric and edge-regular,
then $R(H)$ is normal in $A = \Aut(\G)$, and the stabiliser in $A$ of the vertex $1_0$ is cyclic of order $6$.
\end{lemma}

\begin{proof} 
Let $a$ and $b$ be generators for $H$ satisfying the usual relations $a^n=b^2=1$ and $bab=a^{-1}$,
and let $J$ be the subgroup of $A=\Aut(\G)$ generated by $R(a)$.
Then the orbits of $J$ are the four subsets $H_{0c}$, $H_{0d}$, $H_{1c}$ and $H_{1d}$ defined in the
proof of Theorem~\ref{edge-regular-graphs}.  Also note that $\G$ is bipartite, and since $\G$ is $6$-valent
and edge-regular, we have $|A_{1_0}|=|A_{1_1}|=6$ and $A=R(H)A_{1_0}$, so $|A|=12n$.

\smallskip
We begin by proving that if $J = \lg R(a)\rg$ is a normal subgroup of $A$, then $J$ is self-centralising in $A$.
Note that conjugation of $J$ gives a homomorphism from $A$ to $\Aut(J)$, which is abelian
since $J$ is cyclic, and $J$ is contained in the kernel $C = C_A(J)$, so $A/C$ is abelian, of order dividing $|A/J| = 12$.
On the other hand, the involution $b$ does not centralise $a$, so $|A/C|$ is even, and therefore  $|C| = n, 2n, 3n$ or $6n$.

Suppose $|C|=3n$ or $6n$.  Then $|A/C| = 4$ or $2$, so every Sylow $3$-subgroup of $A$ is contained in $C$.
Let $P$ be any one of them, and take $M=JP$.  This contains $J$ as a central subgroup of order $n$,
so $|M|=3n$, which is odd, and therefore $M$ preserves the bipartition of $\G$.
Also $M/J$ is the only subgroup of $C/J$ of order $3$ (since $|C/J|=6$ or $3$), so $M$ is normal in $A$.
Moreover, since both $J \cap M_{1_0}$ and $J \cap M_{1_1}$ are trivial,
we find that $M = J\times M_{1_0} = J\times M_{1_1} \cong C_n \times C_3$, and hence $M$ is abelian.
Now let $N = \lg M_{1_0},M_{1_1} \rg$.
Then $N$ is abelian, and isomorphic to $C_3 \times C_3$ since $\G$ is edge-regular,
and also $N$ is characteristic in $M$ and hence normal in $A$, and therefore $N$ contains the stabiliser of
of every vertex of $\G$.  It follows that every orbit of $N$ on $V(\G)$ has length $3$.
In particular, the neighbourhood $\G(1_0)$ of the vertex $1_0$ is the union of two such orbits,
namely $\Delta_1 = 1_1^{\, N}$ and $\Phi_1 = x_1^{\, N}$ for some $x \in H \setminus \lg a \rg$.
This, however, implies that if $h_0$ is any other vertex of $\Delta_0 = 1_0^{\, N}$, then every edge incident with $h_0$
lies in the same orbit under $N$ as $\{1_0,1_1\}$ or $\{1_0,x_1\}$, so $1_0$ and $h_0$ have exactly the same
neighbours, and therefore $\G$ is unworthy, so cannot be edge-regular. (Indeed $\G \cong K_{6,6}$.) Thus $|C| \le 2n$.

Next, suppose $|C|=2n$.  Then $C$ is generated by $J$ and an involution that centralises $J$,
so $C$ is cyclic. It follows that $C$ cannot act regularly on both $H_0$ and $H_1$, for otherwise
$\G$ would be a bi-Cayley graph over $C_{2n}$, and by Proposition~\ref{abelian} it would
be vertex-transitive and hence arc-transitive, so not edge-regular.
Without loss of generality, $C$ does not regularly on $H_0$, and then the vertex-stabiliser $C_{1_0}$
is a non-trivial characteristic subgroup of $C$ and therefore normal in $A$, so fixes every vertex of $H_0$.
It follows that $C_{1_0}$ is semi-regular on $H_1$, with orbits of length $2$.
But then the two vertices in each orbit of $C_{1_0}$ on $H_1$ have exactly the same neighbours,
and therefore $\G$ cannot be edge-regular, another contradiction.

Hence the only possibility is that $|C| = n$, in which case $J = C = C_A(J)$.

\smallskip
We now proceed to use similar arguments to show that $R(H) \lhd A$.
To do this, we let $K$ be the core of $R(H)$ in $A$.  Now since $|A:R(H)| = |A_{1_0}|=6$, we see that $A/K$
is isomorphic to a subgroup of $S_6$, and then $R(H)/K$ is isomorphic to a subgroup of $S_5$,
but also $R(H)/K$ is a quotient of the dihedral group $R(H) \cong D_n$ of twice odd order,
so the only possiblilities are $\{1\}$, $C_2$, $D_3$ and $D_5$.

If $R(H)/K \cong C_2$, then since $R(H)$ is dihedral of twice odd order,
we have $K = J = \lg R(a)\rg\cong C_n$, and so $A/K = A/J = A/C_A(J)$, which is abelian,
but then $R(H)/K \lhd A/K$, so $R(H) \lhd A$, contradiction.

Also if $R(H)/K \cong D_5$, then $A/K$ has order $60$ and is a product of $D_5$ and $A_{1_0} \cong C_6$,
and hence is soluble.  Some elementary group theory then shows that $A/K$ has a normal Sylow $5$-subgroup,
but then since $|A:J| = 12$ this Sylow $5$-subgroup must be $J/K$, and so $J \lhd A$, contradiction.

Next, suppose that $R(H)/K \cong D_3$.  Then clearly $K=\lg R(a^3)\rg$, and $A/K$ has order $36$
and is a product of $D_3$ and $A_{1_0} \cong C_6$.  In particular, $J = \lg R(a)\rg \leq C_A(K)$,
but also $J$ is not normal in $A$, for otherwise $K = {\rm core}_A(R(H))$ would contain $J$,
and so $J \ne C_A(K)$, and therefore $|C_A(K)|$ is a proper multiple of $n$, dividing $|A| = 12n$.
If $|K|=1$ then $R(H) \cong R(H)/K \cong D_3$ and it follows that $2n =|H| = 6$ and $\G\cong K_{6,6}$,
again contrary to the assumption that $\G$ is edge-regular. Thus $|K|>1$.
Also $|K| = |R(H)|/6 = n/3$, so $|K|$ is odd, and therefore $K=\lg R(a^3)\rg$ cannot be  centralised by $R(b)$,
which in turn implies that $|C_A(K)|$ divides $|A|/2 = 6n$.
On the other hand, $|C_A(K)| \ne 2n$, for otherwise $J$ would be a normal Hall $2'$-subgroup of $C_A(K)$,
and hence characteristic in $C_A(K)$ and normal in $A$.
Thus $|C_A(K)| = 3n$ or $6n$.
Moreover, $|C_A(K)/K|$ is either $3n/(n/3) = 9$ or $6n/(n/3) = 18$.

Now let $P$ be a Sylow $3$-subgroup of $C_A(K)$, and let $M=KP$.
Then $M/K = KP/K$ is a Sylow $3$-subgroup of $C_A(K)/K$, so $|M/K| = 9$,
and hence is normal in $C_A(K)/K$, and therefore characteristic in $C_A(K)/K$, and hence normal in $A/K$.
(In fact $M/K \cong C_3 \times C_3$, because $A$ is the product of the complementary subgroups
$R(H)$ and $A_{1_0}$, with $|A_{1_0}|=6$.)
Thus $M$ is a normal subgroup of $A$, of order $9|K| = 3n$.
Also $J/K$ (of order $3$) must be contained in $M/K$, so $\lg R(a) \rg = J \leq M$.
Consequently, just as before, 
we find that every orbit of $M$ on $V(\G)$ is one the four orbits $H_{0c}$, $H_{0d}$, $H_{1c}$ and $H_{1d}$ of $J$.
Also the induced subgraph $X$ on $H_{0c}\cup H_{1c}$ is a $3$-valent bi-Cayley graph over $\lg a\rg \cong C_n$,
on which $M$ acts edge-transitively, with $M_{1_0}\cong M_{1_1}\cong C_3$.

Also we note that $X$ is connected.  For suppose $C$ is a component of $X$, with parts $C_0=V(C)\cap H_0$
and $C_1=V(C)\cap H_1$.  Then $|C_0|=|C_1|$, since $C$ is $3$-valent.
Also each $C_i$ is a block of imprimitivity for the action of $A$ on $V(\G)$.
Next let $x$ and $y$ be involutory automorphisms in $A_{1_0}\setminus M_{1_0}$ and $A_{1_1}\setminus M_{1_1}$,
respectively.  Then $C_0^{\,x} = C_0$ and $C_1^{\,x} \subseteq H_{1d}$,
while $C_1^{\,y} = C_1$ and $C_0^{\,y}\subseteq H_{0d}$.
The induced subgraphs on $C_0 \cup C_1^{\,x} = (C_0 \cup C_1)^x$ and $C_0^{\,y} \cup C_1 = (C_0 \cup C_1)^y$
are isomorphic to $C$, and contain edges from $H_{00}$ to $H_{1d}$ and from $H_{0d}$ to $H_{1c}$,
while the induced subgraph on $C_1^{\,x} \cup C_0^{\,y}$ contains edges from $H_{1d}$ to $H_{0d}$.
Hence the induced subgraph on $C_0\cup C_1\cup C_1^{\,x} \cup C_0^{\,y}$ is connected and $6$-valent,
so must be $\G$, and therefore $C=X$.  

This means we can apply Corollary~\ref{cyclic} to $X$, and conclude that $n$ divides $\ell^2+\ell+1$
for some $\ell\in\mz_n$, and hence that $n$ is not divisible by $9$.  It follows that the $3$-part of $|C_A(K)|$
is $9$, so $|P| = 9$, and therefore $P$ is abelian.  Hence also $M = KP$ is abelian, and in particular,
$M$ centralises $J$.  Again, however, this implies that the stabiliser of every vertex of $\G$ is a
subgroup of $N = \lg M_{1_0},M_{1_1} \rg \cong C_3 \times C_3$, and so $\G$ is unworthy,
which contradicts edge-regularity.

Thus $R(H)/K$ is not isomorphic to $D_3$, so must be trivial, and we find $R(H) = K \lhd A$, as required.

Finally, since $J = \lg R(a)\rg$ is characteristic in $R(H)$, it follows that $J \lhd A$,
and then $A/J = A/C_A(J)$, which is abelian, and so $A_{1_0} \cong A/R(H)$ is abelian, of order $6$,
and therefore cyclic.
\end{proof}

We can now give and prove the main theorem in this subsection.

\begin{theorem}\label{th-edge-regular}
Let $\,\G={\rm BiCay}(H,\emptyset, \emptyset, S)$ be a connected $6$-valent bi-Cayley graph
over the dihedral group $H = D_n$ of order~$2n$.
Then $\,\G$ is semisymmetric and edge-regular if and only if $\,\G \cong \G(n, \ld, 6)$ 
for some integer $\ld$ satisfying $\ld^6 \equiv 1$ mod $n$ and  $1+\ld^2+\ld^4 \equiv 0$ mod $n$
but $\ld^3\not\equiv-1$ mod $n$.
\end{theorem}

\begin{proof}
First, if $\G \cong \G(n, \ld, 6)$ where $\ld^3\not\equiv-1$ mod $n$, then by Theorem~\ref{edge-regular-graphs}
we know that $\G$ is edge-regular, so it remains to prove the converse.
So suppose $\G$ is semisymmetric and edge-regular, and let $A=\Aut(\G)$.

By Proposition~\ref{n}, up to graph isomorphism we may assume that $S$ generates $H$, and contains the identity element of $H$,
so that $1_1\in \{s_1 : \, s\in S\} = \G(1_0)$.
Also by Lemma~\ref{edge-regular-normal} we know that $R(H)\unlhd A$ and $A_{1_0}\cong C_6$,
and so by Proposition~\ref{normaliser}, we may take $A_{1_0}=\lg\s_{\a, v}\rg$ for some  $\a\in\Aut(H)$
and $v\in H$.

Now $\s_{\a, v}$ takes $1_1$ to $v_1$, so $\s_{\a, v}^{\, i}$ takes $1_1$ to $(vv^\a v^{\a^2}\ldots\ v^{\a^{i-1}})_1$
for all $i \ge 1$, and as $\s_{\a, v}$ has order $6$, it follows that
$S = \{1\} \cup \{\, vv^\a v^{\a^2}\ldots\ v^{\a^{i-1}} \! : \, 1 \le i \le 5\}$.
In particular, since these elements have to generate $H$, we find that $v$ cannot lie in the maximal cyclic
subgroup $C_n$ of $H = D_n$, so must be an involution.
Next, let $u$ be any generator of the subgroup $C_n$, and suppose that $\a$ takes $u$ to $u^\ld$,
and $v$ to $vu^j$, where $\ld\in\mz_n^{\,*}$ and $j\in\mz_n$. Then it is easy to see that
$$ S \ = \ \{1, v, vv^\a, vv^\a v^{\a^2}, vv^\a v^{\a^2}v^{\a^3}, vv^\a v^{\a^2}v^{\a^3}v^{\a^4}\}
\ = \ \{1, v, u^j, vu^{j\ld}, u^{j(1+\ld^2)}, vu^{j(\ld+\ld^3)}\},
$$
and also that $1_1 = 1_1^{\,\s_{\a, v}^{\, 6}} = (v(vu^{j(\ld+\ld^3)})^\a)_1 = (u^{j(1+\ld^2+\ld^4)})_1$,
so $u^{j(1+\ld^2+\ld^4)} = 1$.
Moreover, since $S$ generates $H$, we see that $j$ must be a unit mod $n$,
so $u^j$ has order $n$, and therefore $1+\ld^2+\ld^4 \equiv 0$ mod~$n$.
It follows that $1 - \ld^6 = (1- \ld^2)(1+\ld^2+\ld^4) \equiv 0$ mod $n$, and so $\ld^6 \equiv 1$ mod $n$.
Also $\ld^2 \not\equiv 1$ mod $n$, for otherwise $0 \equiv 1+\ld^2+\ld^4 \equiv 3$ mod $n$,
which implies $n = 3$, but then $\G \cong K_{6,6}$, which is arc-transitive.

We can now take $a = u^j$ and $b = v$ as our canonical generators for $H = D_n$,
and with these we have $S = \{1, b, a, ba^{\ld}, a^{1+\ld^2}, ba^{\ld+\ld^3}\} = S(n,\ld,6)$.
Also $\ld^3\not\equiv 1$ mod $n$, for otherwise $1_1$ and $(ba^\ld)_1$ have the same
neighbours, and so $\G$ is unworthy and hence cannot be edge-regular.
Similarly $\ld^3\not\equiv -1$ mod $n$, for otherwise $1_0$ and $(ba^\ld)_0$ have the
same neighbours, and again $\G$ cannot be edge-regular.

Thus $\ld^6 \equiv 1$ mod $n$ and $1+\ld^2+\ld^4 \equiv 0$ mod $n$ but $\ld^3\not\equiv 1$ mod $n$,
and $\G\cong\G(n, \ld, 6)$, as required.
\end{proof}

\section{Tetravalent half-arc-transitive bi-$p$-metacirculants}
\label{sec:tetravalent}

In this final section, we consider tetravalent half-arc-transitive graphs that are constructible as normal edge-transitive bi-Cayley graphs over metacyclic $p$-groups. One motivation for this comes from some work of Bouwer \cite{Bouwer1970} in 1970.
Bouwer confirmed Tutte's question~\cite{Tutte1966} about the existence of half-arc-transitive graphs with even valency at least $4$, and the smallest graph in his family is a bi-Cayley graph over a non-abelian metacylic group of order $27$. Another motivation comes from the literature on half-arc-transitive metacirculants of prime-power order. 
An {\em $(m, n)$-metacirculant} is a graph $\G$ of order $mn$ which admits an automorphism $\s$ of order $n$ such that $\lg\s\rg$ is semi-regular on $V(\G)$, and an automorphism $\tau$ normalising $\lg\s\rg$ such that $\tau$ has a cycle of size $m$ on $V(\G)$ and cyclically permutes the $m$ orbits of $\lg\s\rg$.

Metacirculant graphs were introduced by Alspach and Parsons~\cite{AP}, and have many interesting and important properties.
A graph is called a {\em weak metacirculant} if it admits a metacyclic group of automorphisms acting transitively on vertices.
It is easy to see that every metacirculant is a weak metacirculant,
and in 2008, Maru\v si\v c and \v Sparl~\cite[p.368]{MS} asked whether the converse is true or false.

In a recent paper, Li, Song and Wang \cite{LWS} claimed to prove that the converse is false, in a theorem stating that every non-split metacyclic $p$-group with $p$ an odd prime acts transitively on the vertices of some half-arc-transitive $4$-valent graph that is a weak metacirculant but not a metacirculant.
Unfortunately they made a mistake in the first paragraph of their proof of Theorem~1.3 in \cite{LWS}, and that theorem is incorrect, as we will see from Theorem \ref{p-metacirculants} below.

Nevertheless it is still true that not every weak metacirculant is a circulant.
In fact, the $6$-valent bi-Cayley graph on the cyclic group $C_{28}$ that we gave in Example~\ref{exam-half-abelian}
is a half-arc-transitive graph of order $56$ that is a weak metacirculant (with the subgroup $C_7 \times Q_8$ of its automorphism group being also a non-split extension of $C_4$ by $C_{14}$), but not a metacirculant
--- as can be confirmed by an easy computation, with the help of {\sc Magma}~\cite{BCP} if necessary.
Two other examples of order $800$ have also been found very recently by \v Sparl and Anton\v ci\v c
\cite{AS}, in the census of all $4$-valent half-arc-transitive graphs up to order $1000$ created by Poto\v cnik, Spiga and Verret
\cite{PSV}.
An infinite family of $6$-valent examples (generalising Example~\ref{exam-half-abelian}) will be constructed in
\cite{Zhou-Zhang}, using the methods developed in the current paper.



For the remainder of this section, we let $p$ be an odd prime.
Also we need some additional background.
If $G$ is a metacyclic group, then every subgroup $H$ of $G$ is also metacyclic
(for if $M$ is a normal cyclic subgroup of $G$ such that $G/M$ is cyclic,
then $H\cap M$ is a cyclic normal subgroup of $H$, and similarly $H/(H\cap M) \cong HM/M$ is cyclic).
For any group $G$, the unique minimal normal subgroup $N$ of $G$ such that $G/N$ is a $p$-group is denoted by $O^p(G)$. Also if $G$ has a normal $p'$-subgroup $C$ such that $G=PC$ for some Sylow $p$-subgroup $P$ of $G$, then $C$ is called a {\em normal $p$-complement} in $G$.

Now let $G$ be any finite group having a nonabelian metacyclic Sylow $p$-subgroup $P$.
Then by a theorem of Sasaki \cite[Proposition~2.1]{Sasaki-meta}, we find that $N_G(P)\cap O^p(G)=O^p(N_G(P))$, and moreover, if $N_G(P)$ has a normal $p$-complement, then so does $G$.
Then by another theorem of Sasaki \cite[Proposition~2.2]{Sasaki-meta} and a theorem
of Lindenberg \cite{Lindenberg}, on automorphisms of split and non-split metacyclic $p$-groups (respectively),
we obtain the following:

\begin{xca}\label{normal complement}
Let $G$ be a finite group having a nonabelian metacyclic Sylow $p$-subgroup $P$.
If $P$ is non-split, then $G$ has a normal $p$-complement.
On the other hand, if $P$ is split, and therefore a semidirect product $K \rtimes Q$ of cyclic $p$-groups,
then either $G$ has a normal $p$-complement, or $P$ has an automorphism $\beta$ such that
$P\cap O^p(G)=P\cap O^p(N_G(P))=K^\beta.$
\end{xca}

We use Proposition \ref{normal complement} to study tetravalent half-arc-transitive metacirculants of prime-power order.

\begin{theorem}\label{p-metacirculants}
Let $\,\G$ be a connected $4$-valent half-arc-transitive graph of order $p^n$ for some odd prime $p$.
Then $\,\G$ is a weak metacirculant if and only if $\,\G$ is a metacircuant.
\end{theorem}

\begin{proof}
Clearly we need only prove necessity.
So let $G$ be a metacyclic subgroup of $A = \Aut(\G)$ that acts transitively on $V(\G)$,
and let $P$ be a Sylow $p$-subgroup of $G$.  Then $P$ is metacyclic.
On the other hand, since $\G$ is half-arc-transitive and $4$-valent, the stabiliser $A_v$ of any vertex $v\in V(\G)$
is a $2$-group, and hence $p$ cannot divide $|G_v|$.  Then because $|P| = p^n = |V(\G)|$, we find that $P$ is
regular on $V(\G)$, and therefore $\G$ is an edge-transitive Cayley graph for $P$.
In particular, $P$ is non-abelian, for otherwise the inversion automorphism of $P$ gives an
arc-reversing automorphism of $\G$, which is impossible since $\G$ is half-arc-transitive.
Moreover, $P$ is a Sylow subgroup of $A$, complemented by the Sylow $2$-subgroup $A_v$.
On the other hand, $A_v$ is not normal in $A$, for otherwise $A_v$ would fixe every vertex of $\G$,
so $A_v$ would be trivial, but then $\G$ could not be edge-transitive.  Thus $A$ has no normal $p$-complement,
and it follows from Proposition~\ref{normal complement} (applied to $A$ rather than $G$),
that $P$ is a split metacyclic group, and the Cayley graph $\G$ for $P$ is a metacirculant.\end{proof}

As well as contradicting Theorem~1.3 in \cite{LWS}, the above proof shows that a tetravalent half-arc-transitive weak metacirculant of  odd prime-power order is a Cayley graph for a split metacyclic $p$-group.
Hence we may call a tetravalent half-arc-transitive Cayley graph for a metacyclic $p$-group a {\em $p$-metacirculant}.
Analogously, we define a {\em bi-$p$-metacirculant} to be a bi-Cayley graph over a metacyclic $p$-group.
\medskip

The next theorem shows that most tetravalent vertex- and edge-transitive bi-Cayley graphs over non-abelian metacyclic $p$-groups are normal.
To prove it, we need the concept of a quotient graph. If $G$ is a group of automorphisms of a graph $\G$,
and $N$ is a normal subgroup of $G$, then the {\em quotient graph\/} of $\G$ relative to $N$ is defined
as the graph $\G_N$ whose vertices are the orbits of $N$ on $V(\G)$, and with two orbits adjacent if there exists an
edge in $\G$ between vertices in those two orbits.

\begin{theorem}\label{normal-bi-meta}
Let $\,\G$ be a connected tetravalent bi-Cayley graph over a non-abelian metacyclic $p$-group $H$,
where $p$ is an odd prime, and suppose $R(H)$ is a Sylow subgroup of a subgroup $G$ of $\Aut(\G)$
that acts transitively on both the vertices and the edges of $\,\G$.
Then $H$ is a split metacyclic group, and $R(H)$ is normal in $G$.
Moreover, if $p > 3$ then $R(H)$ is normal in $\Aut(\G)$, and so $\,\G$ is a normal bi-Cayley graph.
\end{theorem}

\begin{proof}
We begin by noting that $|V(\G)| = 2|H| = 2p^n,$ where $n\geq 3$ because $H$ is non-abelian.
Also $\G$ is vertex-transitive and and edge-transitive (by the hypothesis on $G$),
and therefore $\G$ is either arc-transitive or half-arc-transitive.
Then since the valency of $\G$ is $4$, the stabiliser $A_v$ in $A = \Aut(\G)$ of any vertex $v$ of $\G$
is a group of order $2^{c}3$ or $2^{c}$ for some $c$,
and accordingly $|A| = |V(\G)||A_v| = 2^{c+1}p^n$ or $2^{c+1}3p^n,$
depending on whether or not $\G$ is $2$-arc-transitive.
In the latter case, $A$ is a $\{2,p\}$-group, and therefore soluble (by Burnside's $p^{\a}q^{\b}$ theorem).
The analogous property holds for the subgroup $G$ of $A$, and so either $G$ acts transitively
on the $2$-arcs of $\G$, or $G$ is a $\{2,p\}$-group, and therefore soluble.

Now suppose $G$ has a normal $p$-complement, say $Q$.  Then the product $QG_v$ is a $p'$-group,
which must also be complementary to $R(H)$, so $|QG_v| = |G:R(H)| = |Q|$ and therefore $G_v \le Q$.
Then since $\G$ is $4$-valent and $G$-edge-transitive and $G$-vertex-transitive,
the quotient graph $\G_Q$ is either $1$-valent or $2$-valent, and hence is a cycle or $K_2$,
so $\Aut(\G_Q)$ is cyclic or dihedral. But $Q$ is the kernel of the action of $G$ on $V(\G_Q)$,
so $R(H) \cong R(H)Q/Q \le G/Q \le \Aut(\G_Q)$, and hence this cannot happen.

Thus $G$ has no normal $p$-complement, and so by Proposition~\ref{normal complement}
(again applied to $G$),  it follows that $H$ is a split metacyclic group.

\medskip
Next, we show that $R(H) \lhd G$.  This is a little complicated, so we assume there is a counter-example,
and proceed in several steps to show that cannot happen.
\smallskip\smallskip

{\bf Step 1}.  We prove that $G$ has no non-trivial normal $2$-subgroup.
\smallskip

Suppose $G$ has a non-trivial normal $2$-subgroup $N$.
Then the quotient graph $\G_N$ has valency $2$ or $4$.
If its valency is $2$, then $\G_N$ is a cycle of order $p^n$, but then $R(H)$ cannot act
faithfully on $\G_N$ (because the $p$-subgroup $R(H)$ is non-cyclic), and so $\G_N$ must have valency $4$.
Now if $v$ is any vertex of $\Gamma$, then $v$ and its four neighbours lie in five different orbits of $N$,
so $N_v$ must fix each neighbour of $v$.  By connectedness, $N_v$ fixes every vertex of $\Gamma$,
and so $N_v$ is trivial, and therefore $N$ is semi-regular on $\G$.
In particular, $|N|$ must divide $|V(\G)| = 2p^n$, and so $|N| = 2$.
Also the order of $\G_N$ is $p^n$, and $N$ is the kernel of the action of $G$ on $\G_N$,
so $\overline{G} = G/N$ is a group of automorphisms of $\G_N$.

Next, let $g$ be the involutory generator of $N$.  Then $g$ cannot preserve the orbits $H_0$ and $H_1$
of $R(H)$, since each has odd size $p^n$, and it follows that $\lg R(H),g \rg = N \rtimes R(H)$ is transitive on vertices.
Moreover, the orbits of $N$ form a system of imprimitivity for $G$ on $V(\G)$,
and it follows that no orbit of $N$ can be contained in $H_0$ or $H_1$, for the same reason.
In turn, this implies that $R(H)$ acts transitively and hence regularly on the vertices of $\G_N$,
so $\G_N$ is a Cayley graph for $R(H) \cong H$.

In particular, since its valency $4$ is less than $2p$, it follows from \cite[Corollary~1.2]{LS}
that $\G_N$ is a normal Cayley graph, with $\overline{R(H)} = R(H)N/N$ normal in $\Aut(\G_N)$.
Hence $R(H)N/N$ is also normal in $G/N$, and so $R(H)N$ is normal in $G$.
But $R(H)$ has index $2$ in $R(H)N$, and is therefore a normal Sylow $p$-subgroup of $R(H)N$,
so is characteristic in $R(H)N$, and hence $R(H)$ is normal in $G$, contradiction.
\medskip

{\bf Step 2}.  We show that every minimal normal subgroup of $G$ is a $p$-group.
\smallskip

Here we make use of \cite[Lemma~3.1]{Zhou-Feng-jamc}, which shows that if $J$ is an arc-transitive
group of automorphisms of a tetravalent connected graph of order $2p^m$ where $m > 1$ (and $p$ is prime),
then every minimal normal subgroup of $J$ is solvable.
In particular, this is true for $G$, if $G$ acts transitively on the arcs of $\G$.
On the other hand, if $G$ does not act arc-transitively on $\G$, then
by our earlier observations, $G$ is a $\{2,p\}$-group, and so $G$ itself is soluble.
Hence in both cases, a minimal normal subgroup of $G$ is soluble, and therefore an elementary abelian group.
But this cannot be a $2$-group (by step 1), and cannot be a $3$-group (for otherwise it
would be generated by an element of order $3$ fixing a vertex),
and thus every minimal normal subgroup of $G$ is a $p$-group.
\medskip

{\bf Step 3}.  Let $M = O_p(G)$ be the largest normal $p$-subgroup of $G$, and consider $C_G(M)$ and  $G/M$.
\smallskip

Let $C=C_G(M)$. Then conjugation of $M$ by $G$ makes $G/C$ isomorphic to a subgroup of $\Aut(M)$.

Now suppose that $G/M$ has a normal $2$-subgroup contained in $CM/M$, say $L/M$.
Then $M \leq L\leq CM$, but since $C\unlhd CM$, and $M$ is a $p$-group, every $2$-subgroup of $CM$ is
contained in $C$, and so every Sylow $2$-subgroup of $L$ is contained in $C=C_G(M)$, and therefore $L=M\times Q$,
where $Q$ is a Sylow $2$-subgroup of $L$.  But now this makes $Q$ characteristic in $L$ and hence normal in $G$,
which is impossible since $G$ has no non-trivial normal $2$-subgroup.
Thus $G/M$ has no such subgroup.

Next let $\Aut^\Phi(M)=\lg\,\a\in \Aut(M)\ |\ g^\a\Phi(M)=g\Phi(M), \,\forall g\in M \rg$, where $\Phi(M)$ is the Frattini subgroup of $M$. Then $\Aut^\Phi(M)$ is a normal $p$-subgroup of $\Aut(M)$, with $\Aut(M)/\Aut^\Phi(M)\leq\Aut(M/\Phi(M))$;
see, for example, \cite[pp.\,81--83]{Federico}.
Also let $K$ be the subgroup of $G$ containing $C$ for which $K/C=(G/C)\cap\Aut^\Phi(M)$.
Then $K/C$ is a normal $p$-subgroup of $G/C$, and $G/K\leq\Aut(M/\Phi(M))$.

Now suppose for the moment that $K$ is a $p$-group. Then $K\leq M \leq R(H)$, because $M = O_p(G)$.
Also $M \ne R(H)$, for otherwise $R(H) \lhd G$, and hence the index of each of $K$ and $M$ in $R(H)$ is divisible by~$p$.
On the other hand, $R(H)$ is metacyclic, and therefore $M$ is metacyclic, or possibly cyclic.
But if $M$ is cyclic, then $M/\Phi(M)\cong C_p$ and so $\,G/K \leq \Aut(M/\Phi(M)) \cong C_{p-1},\,$
and which implies that $K$ is a Sylow $p$-subgroup of $G$, and so $R(H)=K=M \lhd G$, contradiction.
Hence $M$ is a non-cyclic metacyclic $p$-group.  It follows that $M/\Phi(M)\cong C_p\times C_p$,
so $\,G/K\leq\Aut(M/\Phi(M))\cong\Aut(C_p\times C_p)\cong {\rm GL}(2, p),\,$
and then since $|{\rm GL}(2, p)|=(p^2-1)(p^2-p)$ is not divisible by $p^2$, we find that  $|R(H): K|=p$, and so $K=M$.
Thus we have shown that if $K$ is a $p$-group, then $K = M$ and $G/M$ is isomorphic to a subgroup of ${\rm GL}(2, p)$.
In particular, this happens if $C \le M$ (for then $C$ is a $p$-group and hence so is $K$).
\medskip

{\bf Step 4}.  We show that $G$ is $2$-arc-transitive on $\G$.
\smallskip

Suppose that $G$ does not act transitively on the $2$-arcs of $\G$. Then $G$ is a $\{2,p\}$-group.
This implies that $CM = M$, for otherwise if $L/M$ were a minimal normal subgroup of $G/M$ contained
in $CM/M$, then by the maximality of $M$ as a normal $p$-subgroup of $G$, we would find that $L/M$ is
a normal $2$-group of $G/M$, which is impossible by what we showed at the beginning of Step 3.
Thus  $C \leq M$, and hence also $K = M$ and $G/M \lesssim {\rm GL}(2, p)$,
by what we showed at the end of Step 3.  Then since $G/M$ is a $\{2,p\}$-group (but not a $2$-group),
it follows that the image of $G/M$ in ${\rm PGL}(2, p) = {\rm GL}(2, p)/Z({\rm GL}(2, p))$ is isomorphic
to a subgroup of $C_p\rtimes C_{p-1}$, and hence has a cyclic normal subgroup of order $p$.
In turn, since $Z({\rm GL}(2, p)) \cong \mz_p^* \cong C_{p-1}$, this implies that $G/M$ has a cyclic normal
Sylow $p$-subgroup of order $p$.  But that must be $R(H)/M$, and so once again $R(H) \lhd G$, contradiction.
\medskip

{\bf Step 5}.  By considering the quotient graph $\G_M$, we show that $|R(H):M| = p$.
\smallskip

Since $M$ is a proper subgroup of $R(H)$, we know that $M$ has at least $2p$ orbits on $V(\G)$.
Also by the $2$-arc-transitivity of $G$ on $\G$, we know that $M$ is the kernel of the action of $G$ on $V(\G_M)$,
and $G/M$ is a $2$-arc-transitive group of automorphisms of $\G_M$.

Now suppose that $|R(H)/M|>p$.  Then $|V(\G_M)| = 2p^2$, and so we can use \cite[Lemma~3.1]{Zhou-Feng-jamc}
again, to conclude that $G/M$ has a minimal normal subgroup $N/M$ that is soluble.
This cannot be a $p$-group or a $3$-group, for the same reasons as before, and so must be a $2$-group.
Also the number of orbits of $N/M$ of $V(\G_M)$ is $|R(H)/M|>p>2$, and hence the same is true for the
number of orbits of $N$ on $V(\G)$,
and as $G$ is $2$-arc-transitive on $\G$, the quotient graph $\G_N$ has valency $4$.
Accordingly, just as in step 1, we find that $N$ acts semi-regularly on $V(\G)$, and $|N|=2|M|$.
Moreover, $R(H)N/N$ acts regularly on $V(\G_N)$, and so $\G_N$ is a tetravalent Cayley graph
for $R(H)N/N$.  Also $R(H) \cap N = M$, and therefore $R(H)N/N \cong R(H)/(R(H) \cap N) \cong R(H)/M$,
which is a metacyclic $p$-group.
If $R(H)/M$ is abelian, then by \cite[Corollary~1.3]{BFSX} we find that $G_N$ is a normal Cayley
graph, with $R(H)N/N \lhd G/N$, and therefore $R(H) \lhd G$, contradiction.
On the other hand, if $R(H)/M$ is non-abelian, then once again by \cite[Corollary~1.2]{LS}
we have $R(H)N/N\unlhd G/N$, contradiction. Thus $|R(H)/M|=p$.
\medskip

{\bf Step 6}.  This is the last step, in which we show that $|R(H):M| \ne p$.
\smallskip

As $G$ is $2$-arc-transitive on $\G$, and $|R(H)/M|=p$, again we find that $M$ is the kernel
of the action of $G$ on $V(\G_M)$, and that $G/M$ is a $2$-arc-transitive group of automorphisms
of $\G_M$, but this time $|V(\G)| = 2|R(H)/M| = 2p$, and so $\G_M$ is
one of the symmetric graphs of order $2p$ classified in \cite{cheng-2p}.
Indeed from Theorem~4.2 and Table~1 in \cite{cheng-2p}, we can see that there are only three possibilities, as follows:
\\[+4pt]
${}$ \hskip 2cm  $\bullet$ \ $p=5\ $ and $\ A_5\leq G/M\leq S_5\times C_2$, \\
${}$ \hskip 2cm  $\bullet$ \ $p=7\ $ and $\ \PSL(2,7)\leq G/M\leq\PGL(2, 7)$, \\
${}$ \hskip 2cm  $\bullet$ \ $p=13\ $ and $\ \PSL(3,3)\leq G/M\leq\PSL(3, 3). C_2$.
\\[-4pt]

Now if $C\leq M$, then $G/M \lesssim {\rm GL}(2, p)$, which implies that $A_5 \lesssim {\rm GL}(2, 5)$,
or $\PSL(2,7) \lesssim {\rm GL}(2, 7)$, or $\PSL(3,3)\lesssim {\rm GL}(2, 13)$, respectively, but none of these
is possible (as can be shown with the help of Magma~\cite{BCP} if necessary), and it follows that  $M < CM$.

Again let $L/M$ be a minimal normal subgroup of $G/M$ contained in $CM/M$.
By step 3 we know that $L/M$ cannot be a $2$-group, and so $L/M$ must be $A_5$, $\PSL(2,7)$ or $\PSL(3,3)$.
In particular, $L/M$ is non-abelian simple, and then because $L/M=(L/M)'=L'M/M$,
we find that $L'M=L$.

If $L' \ne L$, then $M \not\leq L'$ and so $L'\cap M<M$,  with $L'/(L'\cap M)\cong L'M/M = L/M$,
and therefore $|L'|=|L/M||L'\cap M|$.  Also in each of the three cases listed above, the $p$-part of $|L/M|$ is $p$,
and it follows that every Sylow $p$-subgroup of $L'$ has order $p|L'\cap M| < p|M| =|R(H)|$,
Thus $|L'\cap M| < |R(H)|/p$, and we find that $L'$ has at least $p$ orbits on $V(\G)$.
But also $L'$ is characteristic in $L$ and hence normal in $G$, and so the $2$-arc-transitivity of $G$ on $\G$
implies that $L'$ is semi-regular on $V(\G)$, and so $|L'|$ divides $|V(\G)| = 2|R(H)| = 2p^n$.
This makes $L'$ a $\{2,p\}$-group, and therefore $L'$ is soluble, which is impossible because $L/M$ is non-abelian simple.
 Hence $L = L'$, so $L$ is perfect.

Next, if $M$ is abelian, then $M\leq C_G(M) = C$ and so we may suppose that $L\leq C$,
in which case $M \le L \leq C$ and therefore $M \leq Z(L)$.  Since also $L = L'$, it follows
that $M$ is isomorphic to a subgroup of the Schur multiplier of the simple group $L/M$.
The Schur multipliers of $A_5$, $\PSL(2,7)$ and $\PSL(3,3)$ are all cyclic (of orders $2$, $2$ and $1$),
however, while $M$ is not (since $M/\Phi(M)\cong C_p\times C_p$), contradiction.
Hence $M$ is non-abelian.

To complete this step, we consider the subgroup $C_L(M)$, which is normal in $L$.
The quotient $L/C_L(M)$ is isomorphic to a subgroup of $\Aut(M)$,
which is soluble by \cite[Lemmas~2.4,2.6,2.7]{Sasaki-meta}, and so $L/C_L(M)$ is soluble.
But $L$ itself is not soluble, and it follows that $C_L(M)$ cannot be soluble.
Next, because $L/M$ is simple and $C_L(M)M/M \lhd L/M$, we know that $C_L(M)M = L$ or $M$,
but the latter cannot occur since $M$ is soluble.
Thus $C_L(M)M = L$, and so $C_L(M)/(C_L(M)\cap M) \cong C_L(M)M/M=L/M$.

Also $C_L(M)\cap M \ne M$ since $M$ is non-abelian, and so just as we did above for $L'$,
we find that every Sylow $p$-subgroup of $C_L(M)$ has order $p|C_L(M)\cap M| < p|M| = |R(H)|$,
and hence $C_L(M)$ has at least $p$ orbits on $V(\G)$.  But also $C_L(M) = L\cap C_G(M) \lhd G$,
and so $C_L(M)$ is semi-regular on $V(\G)$. In particular,  $|C_M(L)|$ divides $|V(\G)| = 2|R(H)| = 2p^n$,
and so $C_L(M)$ is soluble, contradiction.
\medskip

This final contradiction eliminates any possibility of a counter-example, and therefore $R(H) \lhd G$.
\medskip

Finally, because $\G$ is $4$-valent, the stabiliser $A_v$ of any vertex is a $\{2, 3\}$-group,
and then since $R(H)$ acts regularly on each part of $\G$, it follows that the index $|A:R(H)|$ is of the form $2^{a}3^{b}$.
Hence if $p > 3$, then $R(H)$ is a Sylow $p$-subgroup of $A$, and so we can take $G = A$,
and find that $R(H) \lhd A$, so that the bi-Cayley graph $\G$ is normal.
\end{proof}

We remark that $R(H)$ is not always normal in $\Aut(\G)$ when $p=3$.
A counter-example is the bi-Cayley graph $\BiCay(H,\emptyset,\emptyset,S)$,
where $H$ is the metacyclic group $\lg\, a,b\ |\ a^9 = b^3 = 1, \, b^{-1}ab=a^4\,\rg \cong C_9 \rtimes_4 C_3$,
and $S = \{1, a, ab, a^4b^2\}$.  In fact, a computation using Magma~\cite{BCP} shows that this
graph of order 54 is $2$-arc-transitive, with automorphism group of order 1296, but is not normal as a bi-Cayley graph.

\smallskip
This example provided the idea for construction of the families of tetravalent
half-arc-transitive bi-$p$-metacirculants appearing in the next two lemmas.
The members of both families are constructed from metacyclic groups of the form
$C_{p^2} \rtimes C_p$ with presentation $\,\lg\, a, b\ |\ a^{p^2}=b^{p}=1, \, b^{-1}ab=a^{1+p} \,\rg,\,$
for odd primes $p$.  Note that for every such $p,$ the centre of a group of this form is the cyclic subgroup of order $p$
generated by $a^p,$ and the elements of order dividing $p$ form the index $p$ subgroup
generated by $a^p$ and $b$.

\begin{lemma}\label{bi-p-metafamily1}
For any odd prime $p$, let $H$ be the metacyclic group $\lg\, a, b\ |\ a^{p^2}=b^{p}=1, \, b^{-1}ab=a^{1+p} \,\rg$
of order $p^3$, and then let $\mathcal{G}_p=\BiCay(H, \emptyset, \emptyset, S)$ where $S=\{1, a^2, a^{p}b^2, a^{2-p}b^2\}.$
Then $\mathcal{G}_p$ is a $4$-valent edge-regular half-arc-transitive bi-$p$-metacirculant 
over $C_{p^2}\rtimes_{1+p}C_p$, and is also a Cayley graph, for all $p$.
\end{lemma}

\begin{proof}
First, it is easy to see that $H$ has an automorphism $\a$ taking $a$ to $a^{-1}$, and $b$ to $b$,
and then $S^\a=\{1, a^{-2}, a^{-p}b^2, a^{p-2}b^2\}=a^{-2}S.\,$
By Proposition~\ref{normaliser}, it follows that $\s_{\a, a^{2}}$ is an automorphism of $\mathcal{G}_p$
that fixes $1_0$, and interchanges $1_1$ with $(a^2)_1$, and $(a^{2-p}b^2)_1$ with $ (a^{p}b^2)_1$.
Next, $a^{-(p+1)}$ has order $p^2$, since $p+1$ is coprime to $p^2$, and $a^{p}$ centralises $b$
(indeed $Z(H) = \lg a^p \rg$), so $a^{p}b$ has order $p$, and it follows that $a' = a^{-(p+1)}$ and $b' = a^{p}b$
satisfy the same relations as $a$ and $b$.  Hence there exists an automorphism $\b$ of $H$ that takes
$a$ to $a^{-(p+1)}$ and $b$ to $a^pb$, and then $S^\b = \{1, a^{-2(1+p)}, a^{p}b^2, a^{p-2}b^2\} =S^{-1}(a^{p}b^2),$
so by Proposition~\ref{normaliser}, we find that $\d_{\b, a^{p}b^2, 1}$ is an automorphism of $\mathcal{G}_p$
that takes $(1_0, 1_1)$ to $((a^{p}b^2)_1, 1_0)$.

In particular, $\d_{\b, a^{p}b^2, 1}$ takes a vertex of $H_0$ to a vertex of $H_1$, so $\lg R(H), \d_{\b, a^{p}b^2, 1}\rg$
is transitive on the vertices of $\mathcal{G}_p$.
Similarly, the orbit of the arc $(1_0,1_1)$ under $\lg \s_{\a, a^{2}}, \d_{\b, a^{p}b^2, 1}\rg$ includes
$(1_0,(a^2)_1)$ and also $((a^{p}b^2)_1,1_0)$ and $((a^{2-p}b^2)_1,1_0)$,
and so  $G = \lg R(H), \s_{\a, a^2}, \d_{\b, a^{p}b^2, 1}\rg$ acts transitively on the edges of $\mathcal{G}_p$.

For $p = 3$, an easy computation with {\sc Magma}~\cite{BCP} shows that $\Aut(\mathcal{G}_3) = G$,
which has order  $108$ ($= 4p^3$), and that $\mathcal{G}_3$ is half-arc-transitive.
Hence we may suppose that $p > 3$.
Then also because $\mathcal{G}_p$ is $4$-valent, the stabiliser $G_v$ of any vertex is a $\{2, 3\}$-group,
and it follows that $R(H)$ is a Sylow $p$-subgroup of $G,$ and hence $R(H)$ is normal in $\Aut(\mathcal{G}_p),$
by Theorem~\ref{normal-bi-meta}.

Now suppose $\mathcal{G}_p$ is arc-transitive.
Then since $R(H) \lhd \Aut(\mathcal{G}_p),$ also $\mathcal{G}_p$ is normal locally arc-transitive,
and so by Proposition \ref{normal-arc-tran}, some automorphism $\g$ of $H$ 
takes $S$ to $S^{-1}$.
To consider this possibility, note that the non-trivial elements $a^2$, $a^{p}b^2$ and $a^{2-p}b^2$
in $S$ have orders $p^2$, $p$ and $p^2$, respectively, and their inverses are  $a^{-2}$, $a^{-p}b^{-2}$
and $(a^{2-p}b^2)^{-1} = b^{-2}a^{p-2} = a^{(p-2)(1+p)^2}b^{-2} = a^{(p-2)(1+2p)}b^{-2} = a^{(-2-3p)}b^{-2}.$
Hence $\g$ takes $(a^2,a^{2-p}b^2)$ to either $(a^{-2},a^{(-2-3p)}b^{-2})$ or $(a^{(-2-3p)}b^{-2},a^{-2}),$
but then $\g$ takes $a^{p}b^2 = a^{2(p-1)}a^{2-p}b^2$ to
either $a^{2(1-p)}a^{(-2-3p)}b^{-2} = a^{-5p}b^{-2}$ or  $a^{2(1-p)}a^{-2} = a^{-2p},$
a contradiction in both cases.
Hence $\mathcal{G}_p$ is not arc-transitive, and is therefore half-arc-transitive.

Next, if the stabiliser in $\mathcal{G}_p$ of the edge $\{1_0,1_1\}$ is non-trivial, then it must contain $\s_{\g,1}$
for some non-trivial automorphism $\g$ of $H$ that preserves $S$, and then $\g$ has to swap $a^2$ with $a^{2-p}b^2,$
but in that case $\g$ takes $a^{p}b^2 = a^{2(p-1)}a^{2-p}b^2$ to $(a^{2-p}b^2)^{p-1}a^{2},$
which is of the form $a^{\xi}b^{-2}$ for some $\xi$, and therefore $\g$ cannot fix $a^{p}b^2$.
Hence $\mathcal{G}_p$ is edge-regular as well.

Finally, recall that the subgroup $J = \lg R(H), \d_{\b, a^{p}b^2, 1} \rg$ act transitively
on the vertices of $\mathcal{G}_p$.  It is also easy to see that $(\d_{\b, a^{p}b^2, 1})^2 = R(a^{p}b^2) \in R(H)$,
and so $(\d_{\b, a^{p}b^2, 1})^{2p} = 1$,  and then because $R(H) \lhd A$ it follows
that $J = \lg R(H), (\d_{\b, a^{p}b^2, 1})^p \rg \cong H \rtimes C_2,$ of order $2p^3,$ and therefore $J$ acts regularly
on $\mathcal{G}_p$.  Thus $\mathcal{G}_p$ is a Cayley graph for $H \rtimes C_2$.
\end{proof}

The second family can be handled in a similar way, but it differs from the first one at a few points.

\begin{lemma}\label{bi-p-metafamily2}
For any prime $p$ congruent to $1$ modulo $4$, let $H$ be the same metacyclic group of order $p^3$
as used in Lemma {\em\ref{bi-p-metafamily1}}, namely $\lg\, a, b\ |\ a^{p^2}=b^{p}=1, \, b^{-1}ab=a^{1+p} \,\rg,$
let $\,s=\frac{1+\ld-\ld p}{2}\,$ where $\ld$ is a square root of $-1$ in $\mz_{p^2}$, and
let $\mathcal{H}_p=\BiCay(H, \emptyset, \emptyset, T)$ where $T=\{1, a, a^{s}b^2, a^{1-s}b^2\}.$
Then  $\mathcal{H}_p$ is a $4$-valent half-arc-transitive bi-$p$-metacirculant 
over $C_{p^2}\rtimes_{1+p}C_p$, but is not a Cayley graph, for all $p$.
\end{lemma}

\begin{proof}

Let $\a$ be the automorphism of $H$ taking $(a,b)$ to $(a^{-1},b)$.
Then $T^\a=\{1, a^{-1}, a^{-s}b^2, a^{s-1}b^2\}=a^{-1}T,\,$
and so by Proposition~\ref{normaliser}, we find that $\s_{\a, a}$ is an automorphism of $\mathcal{H}_p$
that fixes $1_0$, and interchanges $1_1$ with $a_1$, and $(a^{s}b^2)_1$ with $(a^{1-s}b^2)_1$.
Next, we note that $\,s(1\!-\!\ld\!-\!\ld p) \equiv 1\!-\! \ld p\,$ mod $p^2,\,$ since
$$
2s(1\!-\!\ld\!-\!\ld p) \ \equiv \ (1+\ld-\ld p)(1\!-\!\ld\!-\!\ld p)  \ \equiv \ (1\!-\!\ld^2) - 2\ld p  \ \equiv \ 2(1\!-\! \ld p)
\ \hbox{ mod } \, p^2.
$$
Accordingly, we find that $a^{s(1-\ld-\ld p)-1}b^2 = a^{-\ld p}b^2,$ which has order $p$,
and then further, that $a^{\ld(p+1)}$ and $a^{s(1-\ld-\ld p)-1}b^2$ satisfy the same defining relations
as the alternative generators $a$ and $b^2$ for $H$, namely $a^{p^2}=(b^2)^{p}=1$ and $b^{-2}ab^2=a^{1+2p}$.
This implies the existence of an automorphism $\b$ that takes $a$ to $a^{\ld(p+1)},$
and $b^2$ to $a^{s(1-\ld-\ld p)-1}b^2 = a^{-\ld p}b^2$, respectively.
The effect of this automorphism $\b$ on the other two non-trivial elements of $S$ is given by
$$(a^{s}b^2)^{\b} \ = \ a^{\ld(p+1)s}a^{s(1-\ld-\ld p)-1}b^2 \ = \ a^{s-1}b^2 \ \quad \hbox{and} \ $$
$$(a^{1-s}b^2)^{\b} \ = \ a^{(1-s)\ld(p+1)}a^{-\ld p} \ = \  a^{\ld-s\ld(p+1)}b^2
\ = \  a^{\ld+1-\ld p -s }b^2 \ = \ a^{2s-s}b^2 \ = \ a^{s}b^2,$$
with the latter occurring since the displayed congruence above gives $s\ld(p+1)) \equiv s-1+\ld p$ mod $p^2.\,$
This now implies that $T^\b=\{1, a^{\ld(p+1)}, a^{s-1}b^2, a^{s}b^2\} = T^{-1}a^sb^2,$
once it is noted that
$$(a^{1-s}b^2)^{-1}a^sb^2 = b^{-2}a^{2s-1}b^2 = a^{(2s-1)(1+2p)} = a^{(\ld-\ld p)(1+2p)}
= a^{\ld-\ld p+2\ld p} = a^{\ld(p+1)}.$$
Hence by Proposition~\ref{normaliser}, we have an automorphism $\d_{\b, a^{s}b^2, 1}$ of $\mathcal{H}_p$
that takes $(1_0, 1_1)$ to $((a^{s}b^2)_1, 1_0)$.

In particular, $\d_{\b, a^{s}b^2, 1}$ takes a vertex of $H_0$ to a vertex of $H_1$, so $\lg R(H), \d_{\b, a^{s}b^2, 1}\rg$
is transitive on the vertices of $\mathcal{H}_p$,
and the orbit of the arc $(1_0,1_1)$ under $\lg \s_{\a, a}, \d_{\b, a^{s}b^2, 1}\rg$ includes
$(1_0,a_1)$ and also $((a^{s}b^2)_1,1_0)$ and $((a^{1-s}b^2)_1,1_0)$,
and therefore  $G = \lg R(H), \s_{\a, a}, \d_{\b, a^{s}b^2, 1}\rg$ acts transitively on the edges of $\mathcal{H}_p$.
Also $p$ is at least $5$, and $\mathcal{H}_p$ is $4$-valent, so the stabiliser $G_v$ of any vertex is a $\{2, 3\}$-group,
and it follows that $R(H)$ is a Sylow $p$-subgroup of $G,$ and then by Theorem~\ref{normal-bi-meta},
that $R(H)$ is normal in $\Aut(\mathcal{H}_p).$

Now suppose $\mathcal{H}_p$ is arc-transitive.
Then since $R(H) \lhd \Aut(\mathcal{H}_p),$ also $\mathcal{H}_p$ is normal locally arc-transitive,
and hence by Proposition \ref{normal-arc-tran}, there exists an automorphism $\g$ of $H$ 
that takes $S$ to $S^{-1}$.
This time the non-trivial elements $u = a$, $\,v = a^{s}b^2$ and $w = a^{1-s}b^2$ in $S$ (which all have order $p^2$)
satisfy $vw^{-1} = u^{2s-1}$, and it is a relatively straightforward exercise to prove that no permutation of the inverses
of those elements satisfies the analogous relation.
In fact, each of $v^{-1}wu^{2s-1}$ and $w^{-1}vu^{2s-1}$ is a non-trivial power of $a$,
while in the other four cases, the relevant product lies outside $\lg a \rg$.
Hence $\mathcal{H}_p$ is not arc-transitive, and is therefore half-arc-transitive.

Next, if the stabiliser in $\mathcal{H}_p$ of the edge $\{1_0,1_1\}$ is non-trivial, then it must contain $\s_{\g,1}$
for some non-trivial automorphism $\g$ of $H$ that preserves $S$.  But a similar exercise to the one above shows that
no non-trivial permutation of the elements $u$, $v$, $w$ defined above satisfies the analogue of
the relation $vw^{-1} = u^{2s-1}$; in fact $wv^{-1}u^{1-2s}$ is a non-trivial power of $a$,
while in the other four cases, the relevant product lies outside $\lg a \rg$.
Hence no such $\g$ exists, and therefore $\mathcal{H}_p$ is edge-regular.

Finally, let $\d$ be the automorphism $\d_{\b, a^{s}b^2, 1}$ of $\mathcal{H}_p$ referred to earlier.
It is easy to check that $\d^2$ takes $1_0$ to $(a^{s-1}b^2)_0$, and $1_1$ to $(a^{s}b^2)_1$,
so that $\d^2$ has the same effect on as $\mathcal{H}_p$ as $\s_{\a, a}R(a^{s-1}b^2)$.
Since $R(H)$ is normal in $A$ with index $4$ but does not contain $\s_{\a, a}$, it follows that the
quotient $A/R(H)$ is cyclic of order $4$, generated by the image of $\d$.  In particular, $A$ has
a unique subgroup of index $2$ (and order $2p^3$), namely $\lg R(H), \d^2 \rg = \lg R(H), \s_{\a, a} \rg$.
This subgroup preserves the two parts $H_0$ and $H_1$ of $\mathcal{H}_p$, however, so does not act transitively
on vertices, and therefore $\mathcal{H}_p$ cannot be a Cayley graph.
\end{proof}

\medskip

\f{\bf Proof of Theorem~\ref{th-normal-bi-meta}}.
The first part of this theorem was proved in Theorem~\ref{normal-bi-meta},
and the rest follows from Lemmas~\ref{bi-p-metafamily1} and   \ref{bi-p-metafamily2}.\hfill\qed

\section{Proof of Theorem~\ref{normal-edge}}
\label{sec:final}

Theorem~\ref{normal-edge} asserts that every normal edge-transitive bi-Cayley graph $\G$
is either arc-transitive, half- arc-transitive or semisymmetric, and that examples of each kind exist.
The first part is easy to see: if $\G$ is not vertex-transitive, then it is semisymmetric, while if it is vertex-transitive
but not arc-transitive, then it is half-arc-transitive.
The second part follows from Theorem~\ref{th-girth6}, Proposition~\ref{semi-dihe} and Lemmas~\ref{bi-p-metafamily1} and \ref{bi-p-metafamily2}, which show the existence of infinitely many examples in each case.



\bibliographystyle{amsplain}

\end{document}